\documentclass[UKenglish]{article}
\usepackage[utf8]{inputenc}

\usepackage{amssymb}
\usepackage{amsthm}
\usepackage{amsmath}
\usepackage{amsfonts}
\usepackage{mathtools}
\usepackage{tikz-cd}
\usepackage{leftidx}
\usepackage{bbm}
\usepackage{enumitem}
\usepackage{amsthm}
\usepackage{thmtools,thm-restate}
\usepackage{xcolor}
\usetikzlibrary{matrix,arrows,decorations.pathmorphing}
\usepackage{stmaryrd}

\theoremstyle{plain}
\newtheorem{thm}{Theorem}[section] 
\newtheorem{lem}[thm]{Lemma}
\newtheorem{prop}[thm]{Proposition}
\newtheorem{cor}[thm]{Corollary}

\newtheorem{introtheorem}{Theorem}[section]

\theoremstyle{definition}
\newtheorem{defn}[thm]{Definition} 

\newtheorem{rem}[thm]{Remark}
\newtheorem{notation}[thm]{Notation}

\newtheorem*{introdefn}{Definition}

\newtheorem{setup}[thm]{Set-up}

\newcommand{\be}{\begin{enumerate}[label=(\alph*) ,leftmargin=*]}
\newcommand{\ben}{\begin{enumerate}[label=(\arabic*), leftmargin=*]}
\newcommand{\ee}{ \end{enumerate} }

\newcommand{\Ker}{\mathsf{Ker}}
\newcommand{\Coker}{\mathsf{Coker}}

\newcommand{\Hom}{\mathsf{Hom}}
\renewcommand{\mod}{\mathsf{mod}}

\makeatletter
\def\blfootnote{\xdef\@thefnmark{}\@footnotetext}
\makeatother

\title{Torsion pairs and quasi-abelian categories}
\author{Aran Tattar}
\date{\vspace{-4ex}}
\begin{document}


\maketitle
\begin{abstract}
We define torsion pairs for quasi-abelian categories and give several characterisations. We show that many of the torsion theoretic concepts translate from abelian categories to quasi-abelian categories. As an application, we generalise the recently defined algebraic Harder-Narasimhan filtrations to quasi-abelian categories. 
\end{abstract}

\section{Introduction}

Torsion classes were introduced for abelian categories by Dickson \cite{dickson1966torsion} to generalise the notion of torsion and torsionfree groups.  Since then they have been widely studied in various contexts including ($\tau$-)tilting theory \cite{adachi2014tau}, \cite{happel1996tilting}, lattice theory \cite{demonet2017lattice} and, more recently, stability conditions \cite{brustle2018stability}, \cite{treffinger2018algebraic}.

Quasi-abelian categories are a particular class of exact categories  (in the sense of Quillen \cite{quillen1973higher}) whose maximal exact structure (\cite{rump2011maximal}, \cite{sieg2011maximal}) coincides with the class of all short exact sequences in the category (see Definition \ref{defnn}).  As the name suggests, they are a weaker structure than abelian categories. Quasi-abelian categories appear naturally in cluster theory \cite{SHAH2019313}  and in the context of Bridgeland's stability conditions \cite{bridgeland2007stability}.  Of particular interest to us is their appearence in torsion theory as observed by  Bondal \& Van den Bergh \cite{bondal2003generators} and Rump \cite{rump2001almost}:  Each torsion(free) class in an abelian category is quasi-abelian and every quasi-abelian category, $\mathcal{Q}$, appears as the torsionfree class of a `left associated' abelian category $\mathcal{L}_\mathcal{Q}$ and as a torsion class of an abelian category $\mathcal{R}_\mathcal{Q}$.

In this paper we seek to exploit this relationship to define and study torsion classes in quasi-abelian categories by describing torsion classes of quasi-abelian categories in terms of the torsion(free) classes in the associated abelian category.  We note that torsion pairs in pre-abelian and semi-abelian categories, which are weaker structures still than quasi-abelian categories, have been studied in \cite{janelidze2007characterization}. In this more general context, torsion pairs no longer have the well-known characterisations that they have in the abelian set up. In \cite{bourn2006torsion} torsion theory in non-abelian, so-called homological categories has also been considered. 

 Based on a characterisation of torsion pairs in abelian categories \cite{dickson1966torsion}, we define a torsion pair for a quasi-abelian category as follows.
\begin{introdefn} (Definition \ref{torsiondefn})
Let $\mathcal{Q}$ be a quasi-abelian category. A \textit{torsion pair in $\mathcal{Q}$} is an ordered pair $(\mathcal{T}, \mathcal{F})$ of full subcategories of $\mathcal{Q}$ satisfying the following.
\be 
\item $\Hom_\mathcal{Q}(\mathcal{T}, \mathcal{F}) = 0$.
\item For all $M$ in $\mathcal{Q}$ there exists a short exact sequence 
\[
\begin{tikzcd} 0 \arrow[r] &\leftidx{_\mathcal{T}}{M} \arrow[r] & M \arrow[r] & M_\mathcal{F} \arrow[r] & 0 \end{tikzcd} 
\]  with $\leftidx{_\mathcal{T}}{M}  \in \mathcal{T}$ and $M_\mathcal{F} \in \mathcal{F}$. 
\ee In this case we call $\mathcal{T}$ a \textit{torsion class} and $\mathcal{F}$ a \textit{torsionfree class}. 
\end{introdefn} 
 We establish a correspondence between certain torsion pairs in an abelian
category and torsion pairs in a related quasi-abelian category.
 \begin{introtheorem} \label{thmA}(Theorems \ref{qab} and \ref{torsbij1}) Let $\mathcal{A}$ be an abelian category and let $(\mathcal{C}, \mathcal{D})$, $(\mathcal{C}', \mathcal{D}')$ be torsion pairs in $\mathcal{A}$ such that $\mathcal{C}\subseteq \mathcal{C}'$. 
 Then the intersection $\mathcal{C}' \cap \mathcal{D}$ is quasi-abelian and there is an inclusion preserving bijection:
\begin{align*} \{(\mathcal{X}, \mathcal{Y}) \text{ torsion pair in } \mathcal{A}\ |\ \mathcal{C} \subseteq \mathcal{X} \subseteq \mathcal{C}'\} & \longleftrightarrow  \{ (\mathcal{T}, \mathcal{F}) \text{ torsion pair in }\mathcal{C}'\cap \mathcal{D} \} \\ (\mathcal{X}, \mathcal{Y}) & \longmapsto  (\mathcal{X} \cap \mathcal{D}, \mathcal{Y}\cap \mathcal{C}') \\ (\mathcal{C}\ast \mathcal{T}, \mathcal{F} \ast \mathcal{D}') & \longmapsfrom  (\mathcal{T}, \mathcal{F}).
\end{align*} 
\end{introtheorem}
 We remark that this result generalises the bijection in \cite{jasso2014reduction} where functorially finite torsion classes in abelian categories are considered. Furthermore, independently, in the situation of $\mathcal{C}' \cap \mathcal{D}$ being wide in \cite{asai2019wide} it was shown the above bijection and that it induces an isomorphism of lattices. As a conseqeunce of Theorem \ref{thmA}, we obtain their isomorphism of lattices in our more general setting (see Corollary \ref{asai}). We note that our work differs from \cite{asai2019wide} in that our aim is to understand torsion pairs for quasi-abelian categories, therefore we also consider the cases when $\mathcal{C}' \cap \mathcal{D}$ is not wide. We also note that comparable bijections hold in the triangulated setting \cite[Theorems 5.1 and 5.2]{LX}.
 
 With this machinery in hand, our strategy for studying properties of torsion classes in a quasi-abelian category $\mathcal{Q}$ is to translate the problem to the associated abelian category $\mathcal{L}_\mathcal{Q}$ using the above bijection, utilise the properties of torsion in abelian categories, then translate back to $\mathcal{Q}$.  We see that in general, torsion theoretic concepts of abelian categories carry well to quasi-abelian categories. In particular, we show that the following well-known properties of torsion pairs in the abelian setting  hold in the quasi-abelian case. Furthermore, we characterise when $\mathcal{L}_\mathcal{Q}$ is a small module category over a right noetherian (resp. right artinian ring). 
\begin{introtheorem} (Theorem \ref{Lsmallmod} and Propositions \ref{qabinqab}, \ref{qabaddtors} and \ref{ffqab}) Let $\mathcal{Q}$ be a quasi-abelian category. Then $(\mathcal{T}, \mathcal{F})$ is a torsion pair in $\mathcal{Q}$ if and only if $\mathcal{T}^{\perp} = \mathcal{F}$ and $\mathcal{T}= {}^{\perp}\mathcal{F}$. Moreover, the following hold 
\be 
\item $\mathcal{T}$ and $\mathcal{F}$ are both quasi-abelian categories. 
\item If $\mathcal{Q}$ is noetherian (resp. noetherian and artinian) with respect to subobjects and has a strong projective generator, then $\mathcal{L}_\mathcal{Q} \cong \mathsf{mod}\Lambda$ for a right noetherian (resp. right artinian) ring $\Lambda$. If, furthermore, $\Lambda$ is an artin algebra then $\mathcal{T}$ is functorially finite if and only if $\mathcal{F}$ is functorially finite. 
\ee
\end{introtheorem}

 As an application of our results, in Section \ref{hnfilt} we show that quasi-abelian categories admit algebraic Harder-Narasimhan filtrations as recently studied in the abelian context in \cite{treffinger2018algebraic}. Such filtrations were extensively studied in \cite{reineke2003harder} and named after  Harder and Narasimhan for their work \cite{harder1975cohomology}. Furthermore, Rudakov \cite{rudakov1997stability} showed that every stability function on an abelian category induces a Harder-Narasimhan filtration of each object. In \cite{baumann2014affine} and \cite{brustle2018stability} it was observed, for abelian categories, that each stability function induces a chain of torsion classes; and in \cite{treffinger2018algebraic} the above is generalised to show that every chain of torsion classes satisfying mild finiteness conditions in an abelian category induces Harder-Narasimhan filtrations. We show that the same is true for chains of torsion classes in quasi-abelian categories. Namely, we show the following.
 \begin{introtheorem} (Corollary \ref{cor1}) Every chain of torsion classes satisfying mild finiteness conditions (see Definition \ref{CTA}) in a quasi-abelian category induces a  Harder-Narasimhan filtration of each object that is unique up to isomorphism. 
 \end{introtheorem}

This article is organised as follows. In Section two, we translate the characterisations of torsion pairs in the abelian setting to the quasi-abelian case and show that, in this case, not all characterisations remain equivalent. This naturally leads to our choice of definition.
In the third Section, we prove that the heart of twin torsion pairs is quasi-abelian. This provides us with a way to generate examples of quasi-abelian categories that are not naturally arising as torsion(free) classes. 
The fourth Section is devoted to proving the bijection of Theorem \ref{thmA}. We furthermore show that, under mild assumptions, this bijection preserves the functorially finiteness of the torsion(free) classes.
In the fifth Section we recall the construction of $\mathcal{L}_\mathcal{Q}$ due to Schneiders \cite{schneiders1999quasi} and give criteria for when it is equivalent to a small module category. We then use the results of the previous sections to  completely characterise torsion pairs for quasi-abelian categories. 
As an application of the newly developed theory, in the final section we show the existence of Harder-Narasimhan filtrations for chains of torsion classes in a quasi-abelian category. Furthermore, we also explore topological properties of the set of chains of torsion classes in a quasi-abelian category.

\smallskip
\textbf{Acknowledgments.} This work was undertaken as part of the author's PhD studies supported by the EPSRC. The author thanks Sibylle Schroll and Hipolito Treffinger for many helpful discussions, Gustavo Jasso for communicating the proof of Proposition \ref{fftorsbij}(b), and the anonymous referee for their comments.

\section{Defining torsion pairs}  \label{tors}

In this paragraph, we define a torsion pair in a pre-abelian category and compare this definition with other candidate formulations coming from the abelian case. We begin by recalling the definitions of the categories that will form the backdrop for our work. Recall that an additive category is a pointed category enriched in abelian groups that admits all binary products and coproducts.  

\begin{defn} \label{defnn} Let $\mathcal{A}$ be an additive category.
\be 
\item \cite[\S 5.4]{Bucur1968} $\mathcal{A}$ is \emph{pre-abelian} if every morphism in $\mathcal{A}$ admits a kernel and a cokernel.
\item \cite[p168]{rump2001almost} $\mathcal{A}$ is \emph{quasi-abelian} (or \emph{almost abelian}) if it is pre-abelian and if cokernels (resp. kernels) in $\mathcal{A}$ are stable under pullback (resp. pushout). 
\ee

\end{defn}

\begin{rem} We make some observations. 
\be 
\item Equivalently, a pre-abelian category $\mathcal{A}$ is quasi-abelian if the class of all short exact squences in $\mathcal{A}$ forms a Quillen exact structure on $\mathcal{A}$. See \cite{rump2011maximal} and \cite{sieg2011maximal} for examples of pre-abelian categories where this it not the case and see \cite{buhler2010exact} for an exposition of exact categories.
\item Any pre-abelian category has split idempotents. Indeed, every idempotent morphism admits a kernel. \ee
\end{rem}

\begin{notation}
Arrows corresponding to kernels (resp. cokernels) will often be decorated as $\rightarrowtail$ (resp. $\twoheadrightarrow$).
\end{notation}
\begin{defn} \label{torsiondefn}
Let $\mathcal{A}$ be a pre-abelian category. A \textit{torsion pair in $\mathcal{A}$} is an ordered pair $(\mathcal{T}, \mathcal{F})$ of full subcategories of $\mathcal{A}$ satisfying the following.
\be 
\item[(T1)] $\Hom_\mathcal{A}(\mathcal{T}, \mathcal{F}) = 0$.
\item[(T2)] For all $M$ in $\mathcal{A}$ there exists a short exact sequence 
\begin{equation}
\begin{tikzcd} 0 \arrow[r] &\leftidx{_\mathcal{T}}{M} \arrow[r, "i_M", tail] & M \arrow[r, "p_M", two heads] & M_\mathcal{F} \arrow[r] & 0 \end{tikzcd} 
\end{equation}  with $\leftidx{_\mathcal{T}}{M}  \in \mathcal{T}$ and $M_\mathcal{F} \in \mathcal{F}$. 
\ee
In this case we call $\mathcal{T}$ a \textit{torsion class}, $\mathcal{F}$ a \textit{torsionfree class} and the short exact sequence in (T2) is called the \textit{$(\mathcal{T}, \mathcal{F})$-canonical short exact sequence of $M$}.
\end{defn}

\begin{prop} \label{adfunk} Let $\mathcal{A}$ be a pre-abelian category. Then a full subcategory $\mathcal{T} \subseteq \mathcal{A}$ is a torsion class in $\mathcal{A}$ if and only if there exists  a fucntor  $\mathsf{t}: \mathcal{A} \to \mathcal{T}$ that is an idempotent and radical kernel subfunctor of the identity such that $\mathcal{T}=\{M \in \mathcal{A} \mid \mathsf{t}M \cong M \}$. Moreover, in this situation such a functor is a right adjoint to the canonical inclusion $ \mathcal{T} \hookrightarrow \mathcal{A}$. 
\end{prop}

Recall that a functor $F:\mathcal{A} \to \mathcal{A}$ is
\be 
 \item \textit{idempotent} if $F(FM)\cong FM$ for all $M \in \mathcal{A}$.
\item \textit{radical} if $F(\Coker(FM \rightarrowtail M)) \cong 0 $ for all $M \in \mathcal{A}$.
\item a \textit{kernel subfunctor of the identity} if for all $M \in \mathcal{A}$ there exists a morphism $FM \rightarrowtail M$ which is a kernel 
and furthermore that for all $f: M \to N$ in $\mathcal{A}$ the diagram 
\[ \begin{tikzcd} FM \arrow[r, tail] \arrow[d, "Ff"'] & M \arrow[d, "f"] \\ FN \arrow[r, tail] & N \end{tikzcd} \] commutes.
\ee

\begin{proof}[Proof of Proposition \ref{adfunk}.] ($\Rightarrow$) Let $\mathcal{F}$ be the torsionfree class associated to $\mathcal{T}$. We verify that the assignment  $M \mapsto \leftidx{_\mathcal{T}}{M}$ satisfies the conditions above. Firstly, let $f: M \to N $ in $\mathcal{A}$ then, by (T1),  $p_Nfi_M: \leftidx{_\mathcal{T}}{M} \to N_\mathcal{F}$ is zero, hence by the universal property of kernels, there exists a unique $\leftidx{_\mathcal{T}}{f}:\leftidx{_\mathcal{T}}{M} \to \leftidx{_\mathcal{T}}{N}$ such that $i_N ({}_\mathcal{T}f) = fi_M$. It is clear that this defines a functor $\mathcal{A}\to \mathcal{T}$ which is, by construction, a subfunctor of the identity. 
To see that it is idempotent, consider the $(\mathcal{T}, \mathcal{F})$-canonical short exact sequence of $\leftidx{_\mathcal{T}}{M}$ for any $M \in \mathcal{A}$. 
\[ \begin{tikzcd} 0 \arrow[r] &_\mathcal{T}(_{\mathcal{T}}{M}) \arrow[r, tail, "i_{_\mathcal{T}{M}}"] & {_\mathcal{T}}{M} \arrow[r, "p_{_\mathcal{T}{M}}", two heads] & {({_\mathcal{T}}{M})}_\mathcal{F} \arrow[r] & 0 \end{tikzcd} \]
 and observe that $p_{_\mathcal{T}{M}} = 0$ by (T1) therefore $i_{{_\mathcal{T}}{M}}$ is an isomorphism and also ${(_\mathcal{T}{M})}_\mathcal{F} \cong 0$. The fact that  ${_\mathcal{T}}{(-)}$ is radical follows from applying a dual argument to the $(\mathcal{T}, \mathcal{F})$-canonical short exact sequence of $M_\mathcal{F}$.
It remains to check that $\mathcal{T} =\{M \in \mathcal{A}\mid {_\mathcal{T}}{M} \cong M \}$, but this follows from the fact that, for all $M \in \mathcal{T}$, $p_M =0$ by (T1). 

($\Leftarrow$) Let $\mathsf{t}: \mathcal{A} \to \mathcal{T}$ be a functor as in the statement and set $\mathcal{F} = \{ M \in \mathcal{A} \mid \mathsf{t}M \cong 0 \} $. Then as $\mathsf{t}$ is radical, $M / \mathsf{t}M \in \mathcal{F}$, for all $M \in \mathcal{A}$. Thus (T2) is satsified. To verify (T1), let $M \in \mathcal{T}, N \in \mathcal{F}$ and $f:M \to N$ be a morphism in $\mathcal{A}$, then there is a commutative diagram
\[ \begin{tikzcd} \mathsf{t}{M} \arrow[r, "\cong", tail] \arrow[d, shift left, "\mathsf{t}f"] & M \arrow[d, "f"] \\ \mathsf{t}{N} \cong 0 \arrow[r, tail ] &  N \end{tikzcd} \] from which we conclude $f=0$ and (T1) is satisfied.  

The fact that such a $\mathsf{t}$ is a right adjoint follows from the fact that every morphism $T \to M $ with $T \in \mathcal{T}$ and $M \in \mathcal{A}$ factors through $\mathsf{t}M$ by the universal property of the kernel.\end{proof}

As a direct consequence, we justify some of our terminology. 

\begin{cor} Let $\mathcal{A}$ be a pre-abelian category and $(\mathcal{T}, \mathcal{F})$ be a torsion pair on $\mathcal{A}$. Then for all $M \in \mathcal{A}$ the $(\mathcal{T}, \mathcal{F})$-canonical short exact sequence is unique up to isomorphism.
\end{cor}

\begin{prop} \label{additivetors} Let $\mathcal{A}$ be a pre-abelian category and $(\mathcal{T}, \mathcal{F})$ be a torsion pair in $\mathcal{A}$. Then 
\be \item For all $M \in \mathcal{A}$, if $\Hom_\mathcal{A}(M, \mathcal{F})=0$ then $M \in \mathcal{T}$.
\item For all $N \in \mathcal{A}$, if $\Hom_\mathcal{A}(\mathcal{T}, N)=0$ then $N \in \mathcal{F}$.
\ee
\end{prop}
\begin{proof}
Let $M \in A$ be such that $\Hom_\mathcal{A}(M, \mathcal{F})=0$, then $p_M =0$ and $M \cong \leftidx{_\mathcal{T}}{M} \in \mathcal{T}$. The second statement follows by a dual argument.
\end{proof}

We show that the converse of the above statement is true for quasi-abelian categories in Proposition \ref{qabaddtors}.

 \begin{notation}
For a subcategory of $\mathcal{X}$ of an additive category $\mathcal{C}$ we write $\mathcal{X}^{\perp_\mathcal{C}} := \{ Y \in \mathcal{C} \mid \mathsf{Hom}_\mathcal{C}(\mathcal{X}, Y) = 0 \}$. We define ${}^{\perp_\mathcal{C}}{\mathcal{X}}$ dually. When there is no chance of confusion, we drop the superscript $\mathcal{C}$. 
\end{notation}

\begin{prop}
Let $\mathcal{A}$ be a pre-abelian category and  $\mathcal{T}$ be a torsion class in $\mathcal{A}$, then $\mathcal{T}$ is closed under extensions, quotients and coproducts. Dually, every torsionfree class in $\mathcal{A}$ is closed under extensions, subobjects and products.
\end{prop}

Recall that a subcategory $\mathcal{X}$ of a pre-abelian category $\mathcal{A}$ is 
\be
\item \textit{closed under extensions} if for all short exact sequences $0 \to M' \rightarrowtail M \twoheadrightarrow M'' \to 0$ in $\mathcal{A}$ such that $M', M'' \in \mathcal{X}$, then $M \in \mathcal{X}$.
\item \textit{closed under quotients} if for all epimorphisms $M \to N$ in $\mathcal{A}$ such that $M \in \mathcal{X}$, then $N \in \mathcal{X}$. Being \textit{closed under subobjects} is defined dually.
\item \textit{closed under coproducts} if for all families of objects $\{X_i \mid i \in I \}$ in $\mathcal{X}$ indexed by a set $I$ such that the coproduct $\coprod_{i\in I} X_i$ exists in $\mathcal{A}$, then $\coprod_{i\in I} X_i \in \mathcal{X}$. Being \textit{closed under products} is defined dually.
\ee
\begin{proof}
Let $\mathcal{A}$ be a pre-abelian category and $(\mathcal{T}, \mathcal{F})$ be a torsion pair in $\mathcal{A}$. We begin by showing that a torsion class, $\mathcal{T}$, is closed under quotients. Let $e: M \rightarrow N$ be an epimorphism in $\mathcal{A}$ with $M \in \mathcal{T}$. As $\mathsf{Hom}_{\mathcal{A}}(\mathcal{T}, \mathcal{F})=0$, the composition $p_Ne: M \to N_\mathcal{F}$ is zero. Thus, as $e$ is an epimorphism, $p_N$ is zero and $N \cong {}_\mathcal{T}N \in \mathcal{T}$.

We now show that $\mathcal{T}$ is closed under extensions. To this end, let $0 \to M' \rightarrowtail M \twoheadrightarrow M'' \to 0$ be a short exact sequence in $\mathcal{A}$ with $M', M'' \in \mathcal{T}$. Since $\mathsf{Hom}_{\mathcal{A}}(\mathcal{T}, \mathcal{F})=0$ and $M' \in \mathcal{T}$, by the universal property of the kernel and cokernel, there exists a commutative diagram 
\[ \begin{tikzcd} 
0 \arrow[r] & M' \arrow[d] \arrow[r, tail] & M \arrow[d, "1"] \arrow[r, two heads] &M'' \arrow[r] \arrow[d, "f"] & 0
\\ 0 \arrow[r] &\leftidx{_\mathcal{T}}{M} \arrow[r, "i_M", tail] & M \arrow[r, "p_M", two heads] & M_\mathcal{F} \arrow[r] & 0. 
\end{tikzcd} \] As $M'' \in \mathcal{T}$, the morphism $f$ is zero and hence so is $p_M$. Thus $M \cong {}_\mathcal{T}M \in \mathcal{T}$.

Finally, we show that $\mathcal{T}$ is closed under coproducts. Let $\{X_i \mid i \in I \}$ be a family of objects in $\mathcal{T}$ indexed by a set $I$ such that the coproduct $\coprod_{i\in I} X_i$ exists in $\mathcal{A}$ and let $\iota_i: X_i \to \coprod_{i\in I} X_i$ denote the canonical inclusions. Suppose there is a morphism $f:  \coprod_{i\in I} X_i \to F$ for some $F \in \mathcal{F}$. Then, as all $X_i \in \mathcal{T}$ the composition $f\iota_i =0$ for all $i \in I$. Thus $f$ is the unique morphism such that the diagram 
\[\begin{tikzcd} X_i \arrow[r, "\iota_i"] \arrow[dr, "0"'] &  \coprod_{i\in I} X_i \arrow[d, "f"] \\ & F \end{tikzcd} \] commutes for all $i \in I$. Since the zero morphism satisfies this property, we conclude that $f=0$.

The statements for torsionfree classes in $\mathcal{A}$ are proved by dual arguements.
\end{proof}

\begin{rem} \label{counterex1} In general, the converse to the above statement is not true: Let $Q$ be the quiver
\[ \begin{tikzcd} 1 \arrow[r] & 2 \arrow[r] & 3 \end{tikzcd} \] and consider the abelian category $\mathcal{A}=\mod KQ$ whose Auslander-Reiten quiver is given by 
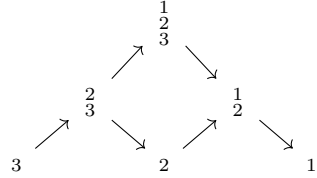
\[\begin{tikzcd}[sep = small] 
& & \begin{smallmatrix} 1 \\ 2 \\ 3 \end{smallmatrix} \arrow[dr] & & \\ & \begin{smallmatrix}  2 \\ 3 \end{smallmatrix} \arrow{ur} \arrow[dr] & & \begin{smallmatrix} 1 \\ 2  \end{smallmatrix} \arrow[dr] & \\ \begin{smallmatrix} 3 \end{smallmatrix}  \arrow[ur] & & \begin{smallmatrix} 2 \end{smallmatrix}  \arrow[ur] & & \begin{smallmatrix} 1 \end{smallmatrix}
\end{tikzcd} \] and consider the subcategory $$ \mathcal{C} = \mathsf{add}\{ \begin{smallmatrix} 3 \end{smallmatrix}\oplus \begin{smallmatrix} 2 \end{smallmatrix} \oplus \begin{smallmatrix} 2 \\ 3 \end{smallmatrix} \oplus\begin{smallmatrix} 1 \\ 2 \end{smallmatrix} \oplus \begin{smallmatrix} 1 \\2 \\3 \end{smallmatrix} \} $$ which is quasi-abelian as it is a torsionfree class of $\mathcal{A}$. Now the subcategory $\mathcal{T}=\mathsf{add}\{ \begin{smallmatrix} 2 \end{smallmatrix} \oplus \begin{smallmatrix} 2 \\ 3 \end{smallmatrix} \} $ of $\mathcal{C}$ is closed under coproducts, extensions and quotients in $\mathcal{C}$ but it is not a torsion class in $\mathcal{C}$. Indeed, $\mathcal{T}^{\perp_\mathcal{C}} = \mathsf{add}\{\begin{smallmatrix} 3 \end{smallmatrix} \}$  but ${}^{\perp_{\mathcal{C}}}(\mathcal{T}^{\perp_\mathcal{C}}) = \mathsf{add}\{ \begin{smallmatrix} 2 \end{smallmatrix} \oplus \begin{smallmatrix} 2 \\ 3 \end{smallmatrix} \oplus\begin{smallmatrix} 1 \\ 2 \end{smallmatrix} \oplus\begin{smallmatrix} 1 \\2 \\3 \end{smallmatrix} \} \neq \mathcal{T}$ which contradicts Proposition \ref{additivetors} thus $\mathcal{T}$ is not a torsion class in $\mathcal{C}$.
\end{rem}

To finish this section, we recall some definitions and make a remark that we will use in the rest of the article. 

\begin{defn} Let $\mathcal{A}$ be a pre-abelian category, $\mathcal{X}\subseteq \mathcal{A}$ be a full subcategory and let $M \in \mathcal{A}$. A \emph{right $\mathcal{X}$-approximation of $M$} is a morphism $\alpha: X \rightarrow M$ with $X \in \mathcal{X}$ such that  all morphisms $X'\rightarrow M$ with $X' \in \mathcal{X}$  factor through $\alpha$:
\[ \begin{tikzcd} X' \arrow[d, dashed, "\exists"'] \arrow[dr, "\forall"] & \\ X \arrow[r, "\alpha"] & M \end{tikzcd} \] 
Dually, we define a  \emph{left $\mathcal{X}$-approximation of $M$}.
The subcategory $\mathcal{X}$ is called \emph{contravariantly finite} (resp. \emph{covariantly finite}) in $\mathcal{A}$ if every $M \in \mathcal{A}$ admits a right (resp. left) $\mathcal{X}$-approximation. $\mathcal{X}$ is called \emph{functorially finite} if it is both contravariantly and covariantly finite in $\mathcal{A}$.
\end{defn}

\begin{rem} \label{torscocov}
It follows immediately from the definitions that for any torsion pair $(\mathcal{T}, \mathcal{F})$ in a pre-abelian category $\mathcal{A}$, $\mathcal{T}$ is a contravariantly finite subcategory with right $\mathcal{T}$-approximations given by the functor ${}_\mathcal{T}(-)$. Dually, $\mathcal{F}$ is a covariantly finite subcategory of $\mathcal{A}$. 
\end{rem}

\section{The heart of twin torsion pairs}

We begin by recalling the structure of torsion(free) classes in abelian categories:
\begin{lem} \cite[Proposition B.3]{bondal2003generators}
\cite[\S 4, Corollary]{rump2001almost}  Every torsion class and torsionfree class of an abelian category has the structure of an quasi-abelian category.  
\end{lem}

In this section we generalise the above result. Namely, we consider the intersection $\mathcal{C}' \cap \mathcal{D}$ where $(\mathcal{C}, \mathcal{D})$, $(\mathcal{C}', \mathcal{D}')$ are torsion pairs in $\mathcal{A}$ such that $\mathcal{C} \subseteq \mathcal{C}'$ or, equivalently, $\mathcal{D}' \subseteq  \mathcal{D}$. We shall refer to such couples of torsion pairs as \textit{twin torsion pairs} and the intersection $\mathcal{C}' \cap \mathcal{D}$ as their \textit{heart}.  We denote twin torsion pairs by $[(\mathcal{C}, \mathcal{D})$, $(\mathcal{C}', \mathcal{D}')]$. We will show that such hearts are quasi-abelian. 

\begin{thm} \label{qab}
 Let $\mathcal{A}$ be an abelian category and let $[(\mathcal{C}, \mathcal{D})$, $(\mathcal{C}', \mathcal{D}')]$ be twin torsion pairs on $\mathcal{A}$. Then the heart, $\mathcal{C}' \cap \mathcal{D}$, is quasi-abelian. 
\end{thm}

\begin{rem} \label{twoofhearts}
We remark that, in general, distinct twin torsion pairs can have the same heart. Indeed, consider the quiver $A_3$ as in Remark \ref{counterex1} and the twin torsion pairs
\[ \Big( \mathsf{add}\{\begin{smallmatrix} 1 \end{smallmatrix} \}, \ \mathsf{add}\{ \begin{smallmatrix} 2 \end{smallmatrix}\oplus \begin{smallmatrix} 3 \end{smallmatrix}\oplus \begin{smallmatrix} 1 \\ 2 \end{smallmatrix}\oplus \begin{smallmatrix} 2 \\ 3  \end{smallmatrix}\oplus \begin{smallmatrix} 1 \\ 2 \\ 3  \end{smallmatrix} \} \Big),\quad \Big( 0 , \  \mathsf{add}\{\begin{smallmatrix} 1 \end{smallmatrix}\oplus \begin{smallmatrix} 2 \end{smallmatrix}\oplus \begin{smallmatrix} 3 \end{smallmatrix}\oplus \begin{smallmatrix} 1 \\ 2 \end{smallmatrix}\oplus \begin{smallmatrix} 2 \\ 3  \end{smallmatrix}\oplus \begin{smallmatrix} 1 \\ 2 \\ 3  \end{smallmatrix} \} \Big) \] and 
\[ \Big( \mathsf{add}\{\begin{smallmatrix} 3 \end{smallmatrix} \}, \ \mathsf{add}\{\begin{smallmatrix} 1 \end{smallmatrix} \oplus \begin{smallmatrix} 2 \end{smallmatrix}\oplus \begin{smallmatrix} 1 \\ 2 \end{smallmatrix} \} \Big),\quad \Big( \mathsf{add}\{\begin{smallmatrix} 1 \end{smallmatrix} \oplus \begin{smallmatrix} 3 \end{smallmatrix} \},\  \mathsf{add}\{\begin{smallmatrix} 1 \end{smallmatrix} \oplus \begin{smallmatrix} 2 \end{smallmatrix}\oplus \begin{smallmatrix} 1 \\ 2 \end{smallmatrix} \} \Big) \] which both have heart $\mathsf{add} \{ \begin{smallmatrix} 1 \end{smallmatrix} \}$. 
\end{rem}

The first step of the proof of Proposition \ref{qab} follows the argument in \cite[Theorem 2]{rump2001almost} and does not require the assumption that the torsion pairs are twin.

\begin{lem} \label{preab} Let $\mathcal{A}$ be an abelian category and let $(\mathcal{C}, \mathcal{D})$, $(\mathcal{C}', \mathcal{D}')$ be torsion pairs in $\mathcal{A}$. Then $\mathcal{C}' \cap \mathcal{D}$ is pre-abelian.
\end{lem}
\begin{proof}
We check the existence of kernels in $\mathcal{C}' \cap \mathcal{D}$, whence existence of cokernels will follow by duality. Let $f: X \to Y$ be a morphism in $\mathcal{C}' \cap \mathcal{D}$, let $g: \Ker f \to X $ be a kernel of $f$ in $\mathcal{A}$ and let $h: {}_{\mathcal{C}'}(\Ker f) \to \Ker f$ be the right $\mathcal{C}'$-approximation of $\Ker f$. Set $K:= {}_{\mathcal{C}'}(\Ker f)$. We claim that $hg: K \to X$ is a kernel of $f$ in $\mathcal{C}' \cap \mathcal{D}$. Firstly, note that $K \in \mathcal{D}$. Indeed, $\mathcal{D}$ is closed under subobjects, and $K$ is a subobject of $\Ker f$ which in turn is a subobject of $X$. 

Now let $u: Z \to X $ be a morphism in $\mathcal{C}' \cap \mathcal{D}$ such that $fu=0$. Then by the universal property of kernels, there exists a unique morphism $v: Z \to \Ker f$ such that $vg=u$. Since $h$ is a right $\mathcal{C}'$-approximation of $\Ker f$ and $Z \in \mathcal{C}'$ there exists a morphism $w: Z \to K$ such that $wh=v$. Together, we have that $u=vg=whg$, thus $u$ factors through $hg$. 
\[ \begin{tikzcd} &  Z \arrow[dr, "u"] \arrow[d, dashed, "\exists v"] \arrow[dl, dashed, "\exists w"'] & & \\ K \arrow[r, "h"'] & \arrow[r, "g"'] \Ker f & X \arrow[r, "f"'] & Y \end{tikzcd} \]
It remains to show that this factorisation is unique. Let $w': Z \to K$ be such that $u=w'(hg)$. Observe that $h$ and $g$ are both monomorphisms and hence so is $hg$. Then $w(hg)=u=w'(hg)$ and we conclude that $w=w'$.
\end{proof}

The previous result shows that kernels (resp. cokernels) in $\mathcal{C}' \cap \mathcal{D}$ are given by kernels in $\mathcal{C}'$ (resp. cokernels in $\mathcal{D}$). 

\begin{notation}
When it exists, we denote the kernel of a morphism $f$ in a subcategory $\mathcal{C}$ of an ambient category $\mathcal{A}$ by $\Ker_\mathcal{C} f$. 
\end{notation}

We observe that the exact structure on $\mathcal{C}' \cap \mathcal{D}$ inherited from $\mathcal{A}$ and the exact structure arising from short exact sequences in $\mathcal{C}' \cap \mathcal{D}$ coincide. 

\begin{prop} \label{exact c'nd}
 Let $\mathcal{A}$ be an abelian category and let $[(\mathcal{C}, \mathcal{D})$, $(\mathcal{C}', \mathcal{D}')]$ be twin torsion pairs in $\mathcal{A}$. Then a pair of composable morphisms $(f, g)$ in $\mathcal{C}' \cap \mathcal{D}$ is a short exact sequence in $\mathcal{C}' \cap \mathcal{D}$ if and only if it is a short exact sequence in $\mathcal{A}$. 
\end{prop}

\begin{proof}
Let $(f: X \to Y, g: Y \to Z)$ be a short exact sequence  in $\mathcal{C}' \cap \mathcal{D}$. Then it follows from Lemma \ref{preab} that $X= \Ker_{\mathcal{C}'} g =  {}_{\mathcal{C}'}(\Ker g)$ and $Z = \Coker_\mathcal{D} f = (\Coker f)_\mathcal{D}$. Consider the commutative diagram with rows that are exact in $\mathcal{A}$

\[ \begin{tikzcd} & & & 0 \arrow[d] & \\ & 0 \arrow[d] & & _\mathcal{C}(\Coker f) \arrow[d, tail] & \\ & X \arrow[r, "f", tail] \arrow[d, tail] & Y \arrow[r, two heads] \arrow[d, "1"] & \Coker f \arrow[r] \arrow[two heads, d] & 0 \\ 0 \arrow[r] & \Ker g \arrow[d, two heads] \arrow[r, tail] & Y \arrow[r, "g", two heads] & Z \arrow[d] \\ & (\Ker g)_{\mathcal{D}'} \arrow[d] & & 0& \\ & 0 & & & \end{tikzcd} \] 

By the Snake lemma we see that $_\mathcal{C}(\Coker f) \cong (\Ker g)_{\mathcal{D}'} \in \mathcal{C} \cap \mathcal{D}'$. But as $\mathcal{C} \subseteq \mathcal{C}'$ we have that $\mathcal{C} \cap \mathcal{D}'=0$. Thus $X \cong \Ker g$, $Z \cong \Coker f$ proving the assertion.

The reverse implication is trivial.
\end{proof}

\begin{proof}[Proof of Theorem \ref{qab}] This follows directly from Lemma \ref{preab} and   Proposition \ref{exact c'nd} since the exact structure inherited from $\mathcal{A}$ consists of all short exact sequences in $\mathcal{C}' \cap \mathcal{D}$.
\end{proof}

\begin{rem} Not every quasi-abelian subcategory of an abelian category arises this way. For example, consider the linearly oriented quiver $Q$ of type $A_3$ as in Remark \ref{counterex1}. Then the subcategory $\mathcal{X} = \mathsf{add}\{ \begin{smallmatrix} 2 \\ 3 \end{smallmatrix} \oplus \begin{smallmatrix} 1  \\2 \end{smallmatrix}\} $ of $\mod KQ$  is quasi-abelian. Indeed, the kernel and cokernel of the morphism $\begin{smallmatrix} 2 \\ 3 \end{smallmatrix} \to \begin{smallmatrix} 1  \\2 \end{smallmatrix}$ are both the zero morphism and there are no non-trivial short exact sequences.  Suppose that $\mathcal{X} = \mathcal{C}' \cap \mathcal{D}$ for some twin torsion pairs $[(\mathcal{C}, \mathcal{D})$, $(\mathcal{C}', \mathcal{D}')]$. Then $ \mathsf{add} \{\begin{smallmatrix} 2  \end{smallmatrix} \}\subset \mathsf{Fac}\mathcal{X} \subseteq \mathcal{C}'$ and $ \mathsf{add} \{\begin{smallmatrix} 2  \end{smallmatrix} \}\subset \mathsf{Sub}\mathcal{X} \subseteq \mathcal{D}$, but $  \mathsf{add} \{\begin{smallmatrix} 2  \end{smallmatrix} \} \not\subset \mathcal{X}$. 
\end{rem}

\section{A bijection of torsion pairs}
In this section, we develop a bijection between the torsion pairs of the heart of two twin torsion pairs and a class of torsion pairs of the ambient category. 
We begin with a technical lemma.

\begin{lem} \label{radicalswitch}
Let $\mathcal{A}$ be an abelian category and let $[(\mathcal{C}, \mathcal{D})$, $(\mathcal{C}', \mathcal{D}')]$ be twin torsion pairs in $\mathcal{A}$. Then for all $M \in \mathcal{A}$, we have 
\be 
\item  $(_{\mathcal{C}'}{M})_\mathcal{D} \cong {}_{\mathcal{C}'}(M_\mathcal{D})=: {}_{\mathcal{C}'}M_\mathcal{D}$. 
\item $_\mathcal{C}(M_{\mathcal{D}'})\cong (_\mathcal{C}M)_{\mathcal{D}'} \cong 0$.
\item $_\mathcal{C}(_{\mathcal{C}'}M) \cong {}_\mathcal{C}{M} \cong {}_{\mathcal{C}'}{(_\mathcal{C}{M})}$.
\item $(M_\mathcal{D})_{\mathcal{D}'} \cong M_{\mathcal{D}'} \cong (M_{\mathcal{D}'})_\mathcal{D}$.
\ee
\end{lem}

\begin{proof}
Let $M \in \mathcal{A}$, using the $(\mathcal{C}, \mathcal{D})$-canonical short exact sequence of $M$ and the $(\mathcal{C}', \mathcal{D}')$-canonical short exact sequence of $M_\mathcal{D}$, we build the following commutative diagram 
\[ \begin{tikzcd} & & & 0 \arrow[d] & \\ 0 \arrow[r] & \leftidx{_\mathcal{C}}{M} \arrow[d, "1"] \arrow[r] & E \arrow[r] \arrow[d, "f"]  \arrow[dr, phantom, "\lrcorner", very near end] & _{\mathcal{C}'}(M_\mathcal{D}) \arrow[r] \arrow[d] & 0 \\ 0 \arrow[r] & \leftidx{_\mathcal{C}}{M} \arrow[r] &M \arrow[r] & M_{\mathcal{D}} \arrow[r] \arrow[d] & 0 \\ & & & (M_\mathcal{D})_{\mathcal{D}'} \arrow[d] & \\ & & &0. &   \end{tikzcd} \]
We make some observations. First note that as $\mathcal{C} \subseteq \mathcal{C}'$ and $\mathcal{C}'$ is closed under extensions, the upper short exact sequence shows that $E \in \mathcal{C}'$. Secondly, by using the Snake Lemma we see that $f$ is a monomorphism and we have a short exact sequence 
\[ \begin{tikzcd} 0 \arrow[r] & E \arrow[r, "f"] & M \arrow[r] & (M_\mathcal{D})_{\mathcal{D}'} \arrow[r] & 0 \end{tikzcd} \] with first term in $\mathcal{C}'$ and last term in $\mathcal{D}'$. Hence, by uniqueness of torsion canonical short exact sequences, we have that $E \cong \leftidx{_{\mathcal{C}'}}{M}$. Now the top row can be written as
\[ \begin{tikzcd} 0 \arrow[r] &\leftidx{_\mathcal{C}}{M} \arrow[r] & \arrow[r]\leftidx{_{\mathcal{C}'}}{M} & \arrow[r] _{\mathcal{C}'}(M_\mathcal{D}) & 0 \end{tikzcd} \] 
which has first term in $\mathcal{C}$ and, as $\mathcal{D}$ is closed under submodules, last term in $\mathcal{D}$. Thus we conclude that $(\leftidx{_{\mathcal{C}'}}{M})_\mathcal{D} \cong {}_{\mathcal{C}'}(M_\mathcal{D})$ and $_\mathcal{C}{(_{\mathcal{C}'}{M})} \cong \leftidx{_\mathcal{C}}{M}$. 

The fact that $_\mathcal{C}(M_{\mathcal{D}'}) \cong 0$ and $(M_\mathcal{D})_{\mathcal{D}'} \cong M_{\mathcal{D}'}$ follows from the commutative diagram with exact rows 
\[ \begin{tikzcd} 0 \arrow[r] & _\mathcal{C}(M_{\mathcal{D}'}) \arrow[d] \arrow[r] & M_{\mathcal{D}'} \arrow[d, "1"] \arrow[r] & (M_\mathcal{D})_{\mathcal{D}'}  \arrow[d] \arrow[r] & 0 \\ 0 \arrow[r] & 0 \arrow[r] &  M_{\mathcal{D}'} \arrow[r, "1"] & M_{\mathcal{D}'} \arrow[r] & 0
\end{tikzcd} \] and the uniqueness of torsion short exact sequences. The remaining isomorphisms are proved similarly. 
\end{proof}

\begin{defn}
Let $\mathcal{A}$ be a pre-abelian category. For two subcategories $\mathcal{X}, \mathcal{Y}$ of $\mathcal{A}$, by $\mathcal{X}\ast \mathcal{Y}$ we denote the subcategory of $\mathcal{A}$ consisting of objects $M \in \mathcal{A}$ for which there exists a short exact sequence
\[ \begin{tikzcd} 0 \arrow[r] & M' \arrow[r, tail] & M \arrow[r, two heads] & M'' \arrow[r] & 0 \end{tikzcd} \] with $M' \in \mathcal{X}$ and $M'' \in \mathcal{Y}$. 
\end{defn}

\begin{rem}
It follows immediately from the definition that if $0 \in \mathcal{X}$ then  $\mathcal{Y} \subseteq \mathcal{X}\ast \mathcal{Y}$ and dually, if $0 \in \mathcal{Y}$ then $\mathcal{X} \subseteq \mathcal{X}\ast \mathcal{Y}$.
\end{rem}

The main result of this section is the following.
\begin{thm} \label{torsbij1}
Let $\mathcal{A}$ be an abelian category and let $[(\mathcal{C}, \mathcal{D})$, $(\mathcal{C}', \mathcal{D}')]$ be twin torsion pairs in $\mathcal{A}$. Then there is an inclusion preserving bijection:
\begin{align*} \{(\mathcal{X}, \mathcal{Y}) \text{ torsion pair in } \mathcal{A}\ |\ \mathcal{C} \subseteq \mathcal{X} \subseteq \mathcal{C}'\} & \longleftrightarrow  \{ (\mathcal{T}, \mathcal{F}) \text{ torsion pair in }\mathcal{C}'\cap \mathcal{D} \} \\ (\mathcal{X}, \mathcal{Y}) & \longmapsto  (\mathcal{X} \cap \mathcal{D}, \mathcal{Y}\cap \mathcal{C}') \\ (\mathcal{C}\ast \mathcal{T}, \mathcal{F} \ast \mathcal{D}') & \longmapsfrom  (\mathcal{T}, \mathcal{F}).
\end{align*} 
\end{thm}
\begin{proof}
We begin by showing the maps are well-defined. First, let $(\mathcal{X}, \mathcal{Y})$ be a torsion pair on $\mathcal{A}$ and suppose that $\mathcal{C} \subseteq \mathcal{X} \subseteq \mathcal{C}'$.  Observe that  $\mathcal{X} \cap \mathcal{D}$ and $\mathcal{Y} \cap \mathcal{C}'$ are  subcategories of $\mathcal{C}'\cap \mathcal{D} $ and we have that $\Hom_{\mathcal{C}'\cap \mathcal{D}}( \mathcal{X} \cap \mathcal{D}, \mathcal{Y} \cap \mathcal{C}')=0$ thus (T1) is satisfied. To verify (T2), let $M \in \mathcal{C}'\cap \mathcal{D}$ and consider the $(\mathcal{X}, \mathcal{Y})$-canonical short exact sequence of $M$
\[ \begin{tikzcd} 0 \arrow[r] &\leftidx{_\mathcal{X}}{M} \arrow[r] & M \arrow[r] & M_\mathcal{Y} \arrow[r] & 0. \end{tikzcd} \] Now as $\mathcal{D}$ is closed under subobjects, $\leftidx{_\mathcal{X}}{M} \in \mathcal{D}$ and thus $\leftidx{_\mathcal{X}}{M} \in \mathcal{X} \cap \mathcal{D}$. Similarly, as $\mathcal{C}'$ is closed under quotients we have that $M_{\mathcal{Y}} \in \mathcal{Y} \cap \mathcal{C}'$. Thus, $(\mathcal{X} \cap \mathcal{D}, \mathcal{Y}\cap \mathcal{C}') $ is a torsion pair in $\mathcal{C}'\cap \mathcal{D} $. 

Conversely, let $(\mathcal{T}, \mathcal{F})$ be a torsion pair in  $\mathcal{C}'\cap \mathcal{D} $. By definition, we have that $\mathcal{C} \subseteq \mathcal{C} \ast \mathcal{T} $ and as $\mathcal{C}'$ is closed under extensions and $\mathcal{T} \subseteq \mathcal{C}'$ we have that $\mathcal{C} \ast \mathcal{T} \subseteq \mathcal{C}'$. Now to show that $(\mathcal{C}\ast \mathcal{T}, \mathcal{F} \ast \mathcal{D}') $ satsifies (T1), let $f: M \to N$ be an arbitrary morphism with $M \in \mathcal{C}\ast \mathcal{T}$ and $N \in  \mathcal{F} \ast \mathcal{D}'$. Consider the diagram
\[ \begin{tikzcd} 0 \arrow[r]& M' \arrow[d, "\exists f'", dashed] \arrow[r] & M \arrow[r] \arrow[d, "f"] & M'' \arrow[r] \arrow[d, "\exists f''", dashed] & 0 \\ 0 \arrow[r] & N' \arrow[r] & N \arrow[r] & N'' \arrow[r] & 0 \end{tikzcd} \] 
where the top (respectively bottom) row shows that $M$ (respectively $N$) is an element of  $\mathcal{C}\ast \mathcal{T}$ (respectively $\mathcal{F} \ast \mathcal{D}'$). That is, $M' \in \mathcal{C}$, $M'' \in \mathcal{T}$, $N' \in \mathcal{F}$ and $N'' \in \mathcal{D}'$. Observe that as $\Hom_\mathcal{A}(\mathcal{C}, \mathcal{D}') =0$, by the universal property of kernels (respectively, cokernels) there exists $f': M' \to N'$ (resp. $f'': M'' \to N''$) rendering the diagram commutative. But since $F \subseteq \mathcal{D}$, we have that $\Hom_\mathcal{A}(\mathcal{C}, \mathcal{F})=0$, so $f'=0$. Similarly, as $T \subseteq \mathcal{C}'$, $f''=0$. By the Snake Lemma there is an exact sequence 
\[ \begin{tikzcd} 0 \arrow[r] & M' \arrow[r] &\Ker f  \arrow[r] & M'' \arrow[r, "\delta"] & N' \arrow[r] & \Coker f \arrow[r] &
 N'' \arrow[r] & 0. \end{tikzcd} \] 
Then $\delta =0$ as $\Hom_{\mathcal{A}}( \mathcal{T}, \mathcal{F}) = \Hom_{\mathcal{C}' \cap \mathcal{D}}(\mathcal{T}, \mathcal{F})=0$. We conclude that $\Ker f \cong M$, $\Coker f \cong N$ and $f=0$.

To show (T2) let $M \in \mathcal{A}$. We begin by using the $(\mathcal{C}', \mathcal{D}')$-canonical short exact sequence of $M$ and the $(\mathcal{C}, \mathcal{D})$-canonical short exact sequence of $\leftidx{_{\mathcal{C}'}}{M}$ to form the pushout of short exact sequences
\[ \begin{tikzcd}
& 0 \arrow[d] & & & \\ 
& \leftidx{_\mathcal{C}}{M} \arrow[d] & & & \\ 0\arrow[r]& \leftidx{_{\mathcal{C}'}}{M} \arrow[r] \arrow[d] \arrow[dr, phantom, "\ulcorner", very near start] &M \arrow[r] \arrow[d] & M_{\mathcal{D}'} \arrow[r] \arrow[d, "1"] & 0 \\ 0 \arrow[r] & \leftidx{_{\mathcal{C}'}}{M}_\mathcal{D} \arrow[r] \arrow[d] & P \arrow[r] & M_{\mathcal{D}'} \arrow[r] &0 \\ & 0. & & & \end{tikzcd} \]
Note that, by the Snake Lemma we have a short exact sequence 
\begin{equation}\label{ses1}
\begin{tikzcd} 0 \arrow[r] &  \leftidx{_\mathcal{C}}{M} \arrow[r] & M \arrow[r] & P \arrow[r] & 0. \end{tikzcd}
\end{equation}
Now, we use the lower short exact sequence of the above diagram and the $(\mathcal{T}, \mathcal{F})$-canonical short exact sequence of $\leftidx{_{\mathcal{C}'}}{M}_\mathcal{D}$ to form the pushout of short exact sequences 
\[ \begin{tikzcd}
& 0 \arrow[d] & & & \\ 
& _\mathcal{T}({_{\mathcal{C}'}}{M}_\mathcal{D}) \arrow[d] & & & \\ 0\arrow[r]& \leftidx{_{\mathcal{C}'}}{M}_\mathcal{D} \arrow[r] \arrow[d] \arrow[dr, phantom, "\ulcorner", very near start] &P \arrow[r] \arrow[d] & M_{\mathcal{D}'} \arrow[r] \arrow[d, "1"] & 0 \\ 0 \arrow[r] & {(_{\mathcal{C}'}{M}_\mathcal{D})}_\mathcal{F} \arrow[r] \arrow[d] & Q \arrow[r] & M_{\mathcal{D}'} \arrow[r] &0 \\ & 0. & & & \end{tikzcd} \]
Then the lower short exact sequence shows that $Q \in \mathcal{F} \ast \mathcal{D}'$ and by the Snake Lemma we have a short exact sequence 
\[\begin{tikzcd} 0 \arrow[r] &  _\mathcal{T}{(_{\mathcal{C}'}{M}_\mathcal{D})} \arrow[r] & P \arrow[r] & Q \arrow[r] & 0. \end{tikzcd} \]
Finally we use this short exact sequence and Sequence (\ref{ses1}) to form the pullback of short exact sequences 
\[ \begin{tikzcd} & & & 0 \arrow[d] & \\ 0 \arrow[r] & \leftidx{_{\mathcal{C}}}{M} \arrow[d, "1"] \arrow[r] & R \arrow[r] \arrow[d, "f"]  \arrow[dr, phantom, "\lrcorner", very near end] & _\mathcal{T}{(_{\mathcal{C}'}{M}_\mathcal{D})} \arrow[r] \arrow[d] & 0 \\ 0 \arrow[r] & \leftidx{_{\mathcal{C}}}{M} \arrow[r] &M \arrow[r] & P\arrow[r] \arrow[d] & 0 \\ & & & Q \arrow[d] & \\ & & &0. &   \end{tikzcd} \]
We observe that the upper short exact sequence shows that $R \in \mathcal{C} \ast \mathcal{T}$. Now by the Snake lemma, there is a short exact sequence 
\[ \begin{tikzcd} 0 \arrow[r] & R \arrow[r] & M \arrow[r] & Q \arrow[r] & 0 \end{tikzcd} \] which shows that (T2) is satisfied. 

We show that the mappings are mutually inverse. Let $(\mathcal{X}, \mathcal{Y})$ be a torsion pair in $\mathcal{A}$ such that $\mathcal{C} \subseteq \mathcal{X} \subseteq \mathcal{C}'$. We claim that $\mathcal{X} = \mathcal{C} \ast (\mathcal{X} \cap \mathcal{D})$. Let $M \in \mathcal{X}$. Observe that, as $\mathcal{C}'$ is closed under quotients and $\mathcal{X}\subseteq \mathcal{C}'$, we have $M_\mathcal{D} \in \mathcal{C}' \cap \mathcal{D}$. Therefore we may build the pullback of short exact sequences using the $(\mathcal{C}, \mathcal{D})$-canonical short exact sequence of $M$ in $\mathcal{A}$ and the $(\mathcal{X}\cap \mathcal{D}, \mathcal{Y} \cap \mathcal{C}')$-canonical short exact sequence of $M_\mathcal{D}$ in $\mathcal{C}' \cap \mathcal{D}$
\[ \begin{tikzcd} & & & 0 \arrow[d] & \\ 0 \arrow[r] & \leftidx{_\mathcal{C}}{M} \arrow[r] \arrow[d, "1"] & E \arrow[d, "f"] \arrow[dr, phantom, "\lrcorner", very near end] \arrow[r] &_{\mathcal{X} \cap \mathcal{D}}{(M_\mathcal{D})} \arrow[r] \arrow[d] & 0 \\ 0 \arrow[r] & _{\mathcal{C}}{M} \arrow[r] & M \arrow[r] & M_\mathcal{D} \arrow[r] \arrow[d] & 0 \\  & & & (M_\mathcal{D})_{\mathcal{Y}\cap \mathcal{C}'} \arrow[d] & \\ & & & 0 & \end{tikzcd} \]
Observe that the upper short exact sequence shows that $E$ is an element of $\mathcal{C} \ast (\mathcal{X} \cap \mathcal{D})$. By the Snake Lemma, we see that $\Coker f \cong (M_\mathcal{D})_{\mathcal{Y}\cap \mathcal{C}'}$. As $M\in \mathcal{X}$, $\Hom_\mathcal{A}(M, \mathcal{Y})=0$ and so $\Coker f =0$. Thus $M \cong E \in \mathcal{C} \ast (\mathcal{X} \cap \mathcal{D})$. The reverse inclusion is clear since both $\mathcal{C}$ and $\mathcal{X} \cap \mathcal{D}$ are contained in $\mathcal{X}$ and $\mathcal{X}$ is closed under extensions. The fact that $\mathcal{Y}= (\mathcal{Y} \cap \mathcal{C}') \ast \mathcal{D}'$ follows by a dual argument. 

Let $(\mathcal{T}, \mathcal{F})$ be a torsion pair in $\mathcal{C}' \cap \mathcal{D}$. We claim that $\mathcal{T}= (\mathcal{C}\ast \mathcal{T}) \cap \mathcal{D}$. Let $M \in (\mathcal{C}\ast \mathcal{T}) \cap \mathcal{D}$. As $M \in \mathcal{C} \ast \mathcal{T}$ there is a short exact sequence \[ \begin{tikzcd} 0 \arrow[r] & C \arrow[r, "f"] & M \arrow[r] & T \arrow[r] & 0 \end{tikzcd} \] with $C \in \mathcal{C}$ and $T \in \mathcal{T}$. Now, as $M \in \mathcal{D}$, $\Hom_\mathcal{A}(\mathcal{C}, M)=0$ and, in particular, $f=0$. Thus $M \cong T \in \mathcal{T}$. The reverse inclusion is clear since $\mathcal{T}\subseteq \mathcal{D}$ by assumption and $\mathcal{T} \subseteq \mathcal{C} \ast \mathcal{T}$ trivially. The fact that $\mathcal{F}= (\mathcal{F} \ast \mathcal{D}') \cap \mathcal{C}'$ holds follows by a dual argument. 
\end{proof}
The following Corollary is a direct consequence of the inclusion preserving property of the bijection in Theorem \ref{torsbij1}. We note that this generalises Theorem 4.2 in \cite{asai2019wide}, where the same result is shown to hold in the case that $\mathcal{C}' \cap \mathcal{D}$ is wide, that is, closed under kernel, cokernels and extensions. 

\begin{cor} \label{asai}  Let $\mathcal{A}$ be an abelian length category and let $[(\mathcal{C}, \mathcal{D})$, $(\mathcal{C}', \mathcal{D}')]$ be twin torsion pairs in $\mathcal{A}$. Then the set of all torsion classes in $\mathcal{C}' \cap \mathcal{D}$ is a complete lattice isomorphic to the lattice interval $[\mathcal{C}, \mathcal{C}']$ of the complete lattice of torsion classes in $\mathcal{A}$.
\end{cor}

The proof of the following Lemma is straightforward and is left to the reader. 

\begin{lem} \label{starstar}
Let $\mathcal{A}$ be an abelian category, $[(\mathcal{C}, \mathcal{D})$, $(\mathcal{C}', \mathcal{D}')]$ be twin torsion pairs in $\mathcal{A}$ and let $\mathcal{T}$ be a torsion class in $\mathcal{C}' \cap \mathcal{D}$. Then for $X \in \mathcal{A}$, we have that $X \in \mathcal{C} \ast \mathcal{T}$ if and only if $X_\mathcal{D} \in \mathcal{T}$. 
In particular, any short exact sequence showing $X$ as an element of $\mathcal{C}\ast \mathcal{T}$ is isomorphic to the $(\mathcal{C}, \mathcal{D})$-canonical short exact sequence of $X$.

\end{lem}

The following Lemma shows that the hearts of twin torsion pairs are preserved under the bijection of Theorem \ref{torsbij1}.

\begin{lem} \label{reflecthearts}
Let $\mathcal{A}$ be an abelian category, $[(\mathcal{C}, \mathcal{D})$, $(\mathcal{C}', \mathcal{D}')]$  twin torsion pairs in $\mathcal{A}$ and let $[(\mathcal{T}, \mathcal{F})$, $(\mathcal{T}', \mathcal{F}')]$ be twin torsion pairs in $\mathcal{C}' \cap \mathcal{D}$. Then \[ \mathcal{T}' \cap \mathcal{F} = (\mathcal{C} \ast \mathcal{T}') \cap (\mathcal{F} \ast \mathcal{D}'). \] 
\end{lem}
\begin{proof}
Let $X \in  (\mathcal{C} \ast \mathcal{T}') \cap (\mathcal{F} \ast \mathcal{D}')$ and consider the commutative diagram \[ \begin{tikzcd} 0 \arrow[r] & _\mathcal{C}X \arrow[r, "i"] \arrow[d, "f"] & X \arrow[r] \arrow[d, "1"] & \arrow[r] \arrow[d, "g"] X_\mathcal{D} & 0 \\ 0 \arrow[r] & _{\mathcal{C}'}X \arrow[r] & X \arrow[r, "q"] & X_{\mathcal{D}'} \arrow[r] & 0 \end{tikzcd} \] with top (resp. bottom) row being the $(\mathcal{C}, \mathcal{D})$-canonical (resp.  $(\mathcal{C}', \mathcal{D}')$-canonical)  short exact sequences of $X$. The existence of the vertical maps $f$ and $g$ follows from the fact that $\mathcal{C} \subseteq \mathcal{C}'$. Moreover, it follows from Lemma \ref{starstar} and its dual that ${}_{\mathcal{C}'}X \in \mathcal{F}$ and $X_\mathcal{D} \in \mathcal{T}$; in particular, ${}_{\mathcal{C}'}X, X_\mathcal{D} \in \mathcal{C}' \cap \mathcal{D}$. Thus, as  $\Hom_\mathcal{A}(\mathcal{C}, \mathcal{D})=0$, $f=0$ and we deduce that $X_\mathcal{C} \cong 0$ and $X \cong X_\mathcal{D}\in \mathcal{T}'$. Similarly, $X \cong {}_{\mathcal{C}'}X \in  \mathcal{F}$ and we have $X \in \mathcal{T}' \cap \mathcal{F}$. The reverse inclusion is trivial.  \end{proof}

\subsection{Functorial finiteness}

In this section, we investigate how the bijection in Theorem \ref{torsbij1} reflects the functorially finite property of torsion(free) classes. Recall that a torsion pair $(\mathcal{T}, \mathcal{F})$ \textit{functorially finite} if both $\mathcal{T}$ and $\mathcal{F}$ are functorially finite subcategories. In the following result, we see how functorially finite torsion pairs behave under the bijection of Theorem \ref{torsbij1}. The proof of part (b) was privately communicated by Gustavo Jasso who  used a similar argument in his work on $\tau$-tilting reduction \cite[Theorem 3.13]{jasso2014reduction}.

\begin{prop} \label{fftorsbij} Let $\mathcal{A}$ be an abelian category and let $[(\mathcal{C}, \mathcal{D})$, $(\mathcal{C}', \mathcal{D}')]$ be twin torsion pairs in $\mathcal{A}$.
\be 
\item If $(\mathcal{X}, \mathcal{Y})$ is a functorially finite torsion pair in $\mathcal{A}$ such that $\mathcal{C}\subseteq \mathcal{X} \subseteq \mathcal{C}'$, then $(\mathcal{X}\cap \mathcal{D}, \mathcal{Y} \cap \mathcal{C'})$ is a functorially finite torsion pair in $\mathcal{C}' \cap \mathcal{D}$.
\item Suppose that $\mathcal{A}$ has enough projectives and injectives and that $(\mathcal{C}, \mathcal{D})$ and $(\mathcal{C}', \mathcal{D}')$ are  functorially finite as torsion pairs in $\mathcal{A}$. Suppose $(\mathcal{T}, \mathcal{F})$ is a functorially finite torsion pair in $\mathcal{C}' \cap \mathcal{D}$, then $(\mathcal{C}\ast \mathcal{T}, \mathcal{F} \ast \mathcal{D}')$ is a functorially finite torsion pair in $\mathcal{A}$.
\ee
\end{prop}

For the proof we will need the following result. \begin{lem} \label{io5.33}\cite[Proposition 5.33]{iyama2013stable} Let $T$ be a triangulated category and $X$ and $Y$ be full subcategories of T . If
$X$ and $Y$ are contravariantly finite in $T$, then so is $X \ast Y$. \end{lem} 

\begin{proof}[Proof of Proposition \ref{fftorsbij}.] (a) Let $(\mathcal{X}, \mathcal{Y})$ be a functorially finite torsion pair in $\mathcal{A}$. In light of Remark \ref{torscocov} we only need to check that $\mathcal{X}\cap \mathcal{D}$ is covariantly finite in $\mathcal{C}' \cap \mathcal{D}$ and that $\mathcal{Y} \cap \mathcal{C}'$ is contravariantly finite in $\mathcal{C}' \cap \mathcal{D}$. We will show the first property, the second will follow by a dual argument.  
Let $M \in \mathcal{C}' \cap \mathcal{D}$ and let $\beta: M \rightarrow X$ be a left $\mathcal{X}$-approximation of $M$, which exists as $M \in \mathcal{A}$. Consider the canonical $(\mathcal{C}, \mathcal{D})$-short exact sequence of $X$
\[ \begin{tikzcd} 0 \arrow[r] & \leftidx{_\mathcal{C}}{X} \arrow[r, "f"] & X \arrow[r, "g"] &  X_\mathcal{D} \arrow[r] & 0. \end{tikzcd} \] Observe that, as $\mathcal{X}$ is closed under factor objects,  $X_\mathcal{D} \in \mathcal{X}$ and therefore $X_\mathcal{D} \in \mathcal{X} \cap \mathcal{D}$. We claim that $g\beta: M \rightarrow X_\mathcal{D}$ is a left $\mathcal{X}\cap \mathcal{D}$-approximation of $M$ in $\mathcal{C}' \cap \mathcal{D}$. 
Indeed, let $r:M \rightarrow X'$ be a morphism with $X' \in \mathcal{X} \cap \mathcal{D}$ then, as $X' \in \mathcal{X}$ and $g$ is a left $\mathcal{X}$-approximation of $M$, there exists a morphism $\gamma: X \rightarrow X'$ such that $\gamma \beta = r$. Now, as $X' \in \mathcal{D}$, $(\gamma f : \leftidx{_\mathcal{C}}{X} \rightarrow X' )= 0$ and as $g =\Coker f$, there exists a morphism $\delta: X_\mathcal{D} \rightarrow X'$ such that $\delta g = \gamma$. Together we have $r= \gamma \beta = \delta (g \beta)$ and thus $r$ factors through $g \beta$ as required.

(b) Suppose that $(\mathcal{C}, \mathcal{D})$ and $(\mathcal{C}', \mathcal{D}')$ are functorially finite torsion pairs in $\mathcal{A}$ and let $(\mathcal{T}, \mathcal{F})$ be a functorially finite torsion pair in $\mathcal{C}'\cap \mathcal{D}$. We claim that $(\mathcal{C}\ast \mathcal{T}, \mathcal{F}\ast \mathcal{D}')$ is a functorially finite torsion pair in $\mathcal{A}$. By Remark \ref{torscocov}, we only need to show that $\mathcal{C}\ast \mathcal{T}$ (resp. $\mathcal{F} \ast \mathcal{D}'$) is covariantly (resp. contravariantly) finite in $\mathcal{A}$. Both facts follow from  Lemma \ref{io5.33} and its dual by using the equivalences $\mathsf{D}^{\operatorname{b}}(\mathcal{A}) \cong \mathsf{K}^-(\mathsf{proj}\mathcal{A})$ and $\mathsf{D}^{\operatorname{b}}(\mathcal{A}) \cong \mathsf{K}^+(\mathsf{inj}\mathcal{A})$ respectively which hold as $\mathcal{A}$ has enough projectives (resp. injectives) together with the observation that, in this case, $\mathcal{A}$ is a functorially finite subcategory of $\mathsf{D}^{\operatorname{b}} (\mathcal{A})$.  \end{proof}  




\section{Torsion pairs in quasi-abelian categories} \label{section:torsqab}

The aim of this section is to characterise torsion pairs in quasi-abelian categories. For a torsionfree class $\mathcal{F}$ of an abelian category $\mathcal{A}$ (which are quasi-abelian as we noted in Section 3) we have already done this in the previous sections: By taking the twin torsion pairs $[(\mathcal{T}, \mathcal{F}), (\mathcal{A}, 0)]$, Theorem \ref{torsbij1} tells us that torsionfree classes in $\mathcal{F}$ are precisely torsionfree classes of $\mathcal{A}$ that lie in $\mathcal{F}$ with corresponding torsion classes obtained by intersecting with $\mathcal{F}$. Using a result of Bondal \& van den Bergh and Rump, we may do this for all quasi-abelian categories.

\begin{lem}\label{QinL}  \cite[Proposition B.3]{bondal2003generators} \cite[Theorem 2]{rump2001almost} Let $\mathcal{Q}$ be a quasi-abelian category. Then $\mathcal{Q}$ is a torsionfree class in an abelian category $\mathcal{L}_\mathcal{Q}$. Moreover, $\mathcal{Q}$ is a cotilting torsionfree class in $\mathcal{L}_\mathcal{Q}$, that is, every object in $ \mathcal{L}_\mathcal{Q}$ is a quotient of an object in $\mathcal{Q}$. 
\end{lem}

$\mathcal{L}_{\mathcal{Q}}$ is sometimes referred to as the \textit{(left) associated abelian category of $\mathcal{Q}$}. 

\subsection{The category $\mathcal{L}_\mathcal{Q}$.}
Following \cite{rump2001modules} and originally due to Schneiders \cite[\S 1.2]{schneiders1999quasi}, we give a construction of $\mathcal{L}_\mathcal{Q}$. Then we investigate the conditions on $\mathcal{Q}$ such that $\mathcal{L}_\mathcal{Q}$ is a (small) module category. 

Recall the homotopy category of $\mathcal{Q}$, $\mathsf{K}(\mathcal{Q})$, whose objects are chain complexes of objects of $\mathcal{Q}$ and morphisms are chain complex morphisms modulo homotopy, see \cite[1.2.1]{schneiders1999quasi} for details. Let $\mathcal{X}$ be the subcategory of $\mathsf{K}(\mathcal{Q})$ consisting of complexes concentrated in degrees $0$ and $1$ with the non-trivial differential being an monomorphism. That is, complexes of the form 
$$ \begin{tikzcd} \dots \arrow[r] & 0 \arrow[r] & X^0 \arrow[r, "f"] & X^1 \arrow[r] & 0 \arrow[r] & \dots \end{tikzcd} $$ that are exact in $X^0$. In practice, we identify the above complex with the monomorphism $f$.

\begin{rem} \label{rem} We make some observations.
\be 
\item A morphism,  $(\alpha, \beta): f \to f'$  in $\mathcal{X}$ is just a  commutative square
\begin{equation} \label{morph}
\begin{tikzcd} X \arrow[r, "f"] \arrow[d, "\alpha"] & Y \arrow[d, "\beta"] \\ X' \arrow[r, "f'"] & Y' \end{tikzcd} 
\end{equation} and is null homotopic if there exists $h: Y \to X'$ in $\mathcal{Q}$ such that $hf= \alpha$  and $f'h=\beta$. Observe that, as $f$ is monic, if $f'h=\beta$   then   $hf= \alpha$ is automatically satisfied.
\item \label{Rkercoker}\cite[Proposition 6]{rump2001modules} We may describe kernels and cokernels of a morphism as in (\ref{morph}) explicitly in $\mathcal{X}$. Consider the commutative diagrams in $\mathcal{Q}$
\[ 
\begin{tikzcd} X \arrow[dr, dashed, "\exists ! r"] \arrow[drr, "f", bend left] \arrow[ddr, "\alpha", bend right] & & \\ & A  \arrow[r, "u"] \arrow[d, "v"] \arrow[dr, phantom, "\lrcorner", very near end]  & Y \arrow[d, "\beta"]  & \\ & X'  \arrow[r, "f'"] & Y'  \\ & & &  \end{tikzcd} \hspace{1cm}
\begin{tikzcd} & & & \\ A \arrow[r, "u"] \arrow[d, "v"] & Y \arrow[d, "q"] \arrow[ddr, bend left, "\beta"] & \\ X' \arrow[drr, "f'", bend right] \arrow[r, "p"] & B \arrow[ul, phantom, "\ulcorner", very near end] \arrow[dr, " \exists ! s", dashed] & \\ & & Y' \end{tikzcd}  
\] Then it is easily verified that the morphisms in $\mathcal{X}$ 
\[ \begin{tikzcd} X \arrow[r, "r"] \arrow[d, "1"] & A \arrow[d, "u"] \\ X \arrow[r, "f"] & Y \end{tikzcd} \hspace{1cm} 
\begin{tikzcd} X' \arrow[r, "f'"] \arrow[d, "p"] & Y' \arrow[d, "1"] \\ B \arrow[r, "s"] & Y' \end{tikzcd} 
\] give a kernel and cokernel of (\ref{morph}) in $\mathcal{X}$ respectively. It follows that a morphism as in (\ref{morph}) is monic if and only if it is a pullback and it is regular (that is, both monic and epic) if and only if it is an exact square in $\mathcal{Q}$. Futhermore, we can naturally decompose any morphism as in (\ref{morph}) into a cokernel followed by a regular morphism followed by a kernel: 
\[ 
\begin{tikzcd} X \arrow[r, "f"] \arrow[d, "r"] &
 Y \arrow[d, "1"] \\ A \arrow[r, "u"] \arrow[d, "v"] \arrow[dr, phantom, "\lrcorner", very near end] & Y \arrow[d, "q"] \\ X' \arrow[d, "1"] \arrow[r, "p"] & B \arrow[ul, phantom, "\ulcorner", very near end] \arrow[d, "s"]  \\ X' \arrow[r, "f'"]  & Y'. \end{tikzcd} 
\] 
\ee
\end{rem}

By \cite[Proposition 1, Proposition 3]{rump2001modules}, we may formally invert all regular morphisms to obtain the category $\mathcal{L}_\mathcal{Q}$ which is abelian by \cite[3.2]{freyd1966representations}. 

\begin{rem}
Schneiders originally observed $\mathcal{L}_\mathcal{Q}$ as the heart of a canonical t-structure on the category $\mathsf{K}(\mathcal{Q})$ as part of his work to study the derived category of a quasi-abelian category. We also note that $\mathcal{L}_\mathcal{Q}$ is a special case of \cite[Exemple 1.3.22]{faisceaux} where the authors made the corresponding construction in the more general setting of exact categories. Similar categories have also been considered in other contexts, for example \cite{wegner2017heart}.
\end{rem}


\begin{rem} \label{inclusion}
 There is a canonical inclusion \[ \begin{array}{rcl} \mathcal{Q} & \hookrightarrow & \mathcal{L}_\mathcal{Q} \\ x & \mapsto & (0 \to x) \end{array} \] which is full, faithful, additive, exact and preserves monomorphisms. We implicitly identify $\mathcal{Q}$ with its image in $\mathcal{L}_\mathcal{Q}$. 
\end{rem}

\begin{defn}
An object $P$ in a quasi-abelian category $\mathcal{Q}$ is a \emph{projective} if for all cokernels $f:X \twoheadrightarrow Y$, every morphism $P \to Y$ factors through $f$:
\[\begin{tikzcd} & P \arrow[d, "\forall"] \arrow[dl, dashed, "\exists"'] \\ X \arrow[r, two heads, "f"']& Y. \end{tikzcd} \] 
By $\mathsf{Proj}\mathcal{Q}$ we denote the subcategory of projective objects in $\mathcal{Q}$. An object $P \in \mathsf{Proj}\mathcal{Q}$ is a \emph{strong projective generator} if it is projective and for all $X \in \mathcal{Q}$ there exists a cokernel $P^I \twoheadrightarrow X$ for some set $I$.
\end{defn}

\begin{lem} \label{projLprojQ} \cite[Lemma 4]{rump2001modules}
Let $\mathcal{Q}$ be a quasi-abelian category. Then $\mathsf{Proj}\mathcal{Q} = \mathsf{Proj}\mathcal{L}_\mathcal{Q}$.
\end{lem}

\begin{cor} \label{projgen}
Let $\mathcal{Q}$ be a quasi-abelian category. Then $\mathcal{L}_\mathcal{Q}$ has a (small) projective generator if and only if $\mathcal{Q}$ has a (small) strong projective generator.
\end{cor}
\begin{proof}
Let $(0 \to P)$ be a projective generator of $\mathcal{L}_\mathcal{Q} =: \mathcal{L}$. Then, for all $X \in \mathcal{Q}$ there exists a morphism $p: P^I \to X$ in $\mathcal{Q}$ for some set $I$ such that 
\[ \begin{tikzcd} 0 \arrow[r] \arrow[d] & P \arrow[d, "p"] \\ 0 \arrow[r] & X \end{tikzcd}
\] is an epimorphism in $\mathcal{L}$. By computations as in Remark \ref{rem}\ref{Rkercoker}, the cokernel is 
\[ \begin{tikzcd} 0 \arrow[r] \arrow[d] & X \arrow[d, "1"] \\ \mathsf{Coim}p \arrow[r, "\tilde{p}"] & X. \end{tikzcd} \] By assumption, this morphism is null-homotopic, so there exists a morphism $h: X \to \mathsf{Coim}p$ such that $\tilde{p}h = 1_X$. In particular, $\tilde{p}$ is a retraction. Now, as retractions are cokernels, there is a composition of cokernels in $\mathcal{Q}$
\[ \begin{tikzcd} P^I \arrow[r] & \mathsf{Coim}p \arrow[r, "\tilde{p}"] & Z \end{tikzcd} \] which is necessarily again a cokernel  in $\mathcal{Q}$.

Conversely, suppose that $P$ is a strong projective generator of $\mathcal{Q}$. We will show that $(0 \to P^I)$ is a projective generator of $\mathcal{L}$. By Proposition \ref{projLprojQ}, $(0 \to P)$ is projective, it remains to show that for all $(f:X \to Y) \in \mathcal{L}$ there exists an epimorphism  $(0 \to P) \twoheadrightarrow (f: X \to Y)$. To this end, let $p: P \twoheadrightarrow Y$ be a cokernel in $\mathcal{Q}$. We will show that the composition \[ \begin{tikzcd} 0 \arrow[r] \arrow[d] & P \arrow[d, "p"] \\ 0 \arrow[r] \arrow[d] & Y \arrow[d, "1"] \\ X \arrow[r, "f"] & Y \end{tikzcd} \] is epic. Since the embedding $\mathcal{Q} \hookrightarrow \mathcal{L}$ is exact, the upper commutative square is an epimorphism. By computations as in \ref{rem}\ref{Rkercoker}, the cokernel of the lower commutative square is \[ \begin{tikzcd} X \arrow[r, "f"] \arrow[d, "f"] & Y \arrow[d, "1"] \\ Y \arrow[r, "1"] & Y \end{tikzcd} \] which is null-homotopic by the identity morphism $Y \to Y$. Thus the lower commutative square is an epimorphism. We conclude that the composition is an epimorphism. 
\end{proof}

As a consequence, there is a `Gabriel-Mitchell theorem for quasi-abelian categories', giving conditions on $\mathcal{Q}$ such that $\mathcal{L}_\mathcal{Q}$ is a module category.

\begin{prop} \cite[Proposition 12]{rump2001modules}
Let $\mathcal{Q}$ be a quasi-abelian category. Then $\mathcal{L}_\mathcal{Q} \cong \mathsf{Mod}\Lambda$ if and only if $\mathcal{Q}$ has a small strong projective generator $P$. Moreover, in this case $\Lambda \cong \mathsf{End}_{\mathcal{Q}}P$. 
\end{prop}

We may also describe when $\mathcal{L}_\mathcal{Q}$ is a small module category over certain classes of rings.

\begin{lem} \label{Noethsubobj}
Let $\mathcal{Q}$ be a quasi-abelian category. Then $\mathcal{L}_\mathcal{Q}$ is noetherian (resp. artinian) with respect to subobjects (that is, every ascending (resp. descending) chain of subobjects stabilises) if and only if $\mathcal{Q}$ is noetherian (resp. artinian) with respect to subobjects.
\end{lem}
\begin{proof}
Since $\mathcal{L}_\mathcal{Q}$ is abelian, we may identify subobjects with kernels, thus by \ref{rem}\ref{Rkercoker} we may assume all subobjects of $(f: X \to Y)$ are of the form \[ \begin{tikzcd} X \arrow[r, "r"] \arrow[d, "1"] & A \arrow[d, "u"] \\ X \arrow[r, "f"] & Y \end{tikzcd} \] and note that $r$ and $u$ are monomorphisms.  Hence, an ascending chain of  $f$ subobjects in $\mathcal{L}_\mathcal{Q}$ corresponds to an ascending chain of subobjects 
\[ X \subset A_1   \subset A_2 \subset \dots \subset Y \] in $\mathcal{Q}$ and the claim follows. The dual argument holds for descending chains of subobjects.
\end{proof}

\begin{thm} \label{Lsmallmod}
Let $\mathcal{Q}$ be a quasi-abelian category. Then $\mathcal{L}_\mathcal{Q} \cong \mathsf{mod} \Lambda$ for a right noetherian (resp. artinian ring) if and only if $\mathcal{Q}$ is noetherian (resp. noetherian and artinian) with respect to subobjects and has a strong projective generator $P$. Moreover, in this case $\Lambda \cong \mathsf{End}_\mathcal{Q}(P)$.
\end{thm}
\begin{proof}
By, for example, \cite[Theorem 1 and Corollary 1]{lam1966category} an abelian category is equivalent to a small module category over a right noetherian (resp. artinian ring) if and only if it admits a projective generator and is noetherian (resp. noetherian and artinian) with respect to subobjects; whence the claim follows from Corollary \ref{projgen} and Lemma \ref{Noethsubobj}.
\end{proof}

\begin{rem}
Dually, for a quasi-abelian category $\mathcal{Q}$, one may construct an abelian category $\mathcal{R}_\mathcal{Q}$ such that $\mathcal{Q}$ is a torsion class in $\mathcal{R}_\mathcal{Q}$. This gives another proof of Proposition \ref{qab}. Moreover, the categories $\mathcal{L}_\mathcal{Q}$ and $\mathcal{R}_\mathcal{Q}$ are derived equivalent and are related by tilting induced by $\mathcal{Q}$, see \cite{bondal2003generators}, \cite{fiorot2016n}, \cite{rump2001modules} and \cite{schneiders1999quasi} for more details.
\end{rem}

\subsection{Properties of torsion(free) classes}

We compare the properties of torsion and torsion free classes in quasi-abelian categories with the abelian setting.

\begin{prop} \label{qabinqab}
Every torsion  and torsionfree class in a quasi-abelian category has the structure of a quasi-abelian category.
\end{prop}
\begin{proof} 
By Theorem \ref{torsbij1} every torsionfree class, $\mathcal{F}$, of a quasi-abelian category $\mathcal{Q}$ is a torsion class of $\mathcal{L}_\mathcal{Q}$ that happens to lie in $\mathcal{Q}$ and is therefore quasi-abelian.  Theorem \ref{torsbij1} also tells us that the associated torsion class in $\mathcal{Q}$ is $\mathcal{T} \cap \mathcal{Q}$ where $\mathcal{T} = {}^\perp \mathcal{F}$ in $\mathcal{L}_\mathcal{Q}$. But $\mathcal{T} \cap \mathcal{Q}$ is the intersection of a torsion and torsionfree class, and as $\mathcal{F} \subseteq \mathcal{Q}$, by Proposition \ref{qab}, it is quasi-abelian. 
\end{proof}

There is an immediate consequence of Lemma \ref{reflecthearts}. 
\begin{cor}
The heart of twin torsion pairs in a quasi-abelian category is quasi-abelian.
\end{cor}

We now prove that the converse of Proposition \ref{additivetors} holds in quasi-abelian categories giving a familiar characterisation of torsion classes. 

\begin{prop} \label{qabaddtors} Let $\mathcal{Q}$ be a quasi-abelian category. Then a pair of full subcategories $(\mathcal{T}, \mathcal{F})$ is a torsion pair on $\mathcal{Q}$ if and only if the following hold.
\be \item For all $M \in \mathcal{Q}$, if $\Hom_\mathcal{Q}(M, \mathcal{F})=0$ then $M \in \mathcal{T}$.
\item For all $N \in \mathcal{Q}$, if $\Hom_\mathcal{Q}(\mathcal{T}, N)=0$ then $N \in \mathcal{F}$.
\ee
\end{prop}
\begin{proof}
The fact that the conditions are necessary was proved in Proposition \ref{additivetors}. We now prove that they are sufficient. Let $\mathcal{T}, \mathcal{F}$ be full subcategories of $\mathcal{Q}$ such that $$\begin{array}{rl}
\mathcal{T} =& \{M \in \mathcal{Q}\ | \ \Hom_\mathcal{Q}(M, \mathcal{F})=0 \} \\ \mathcal{F} = & \{ N \in \mathcal{Q} \ | \ \Hom_\mathcal{Q}(\mathcal{T}, N) =0 \}. 
\end{array} $$ Observe that if $({}^\perp \mathcal{Q} \ast \mathcal{T}, \mathcal{F} )$ is a torsion pair in $\mathcal{L}=\mathcal{L}_\mathcal{Q}$, then $(({}^\perp \mathcal{Q} \ast \mathcal{T}) \cap \mathcal{Q}, \mathcal{F}) = (\mathcal{T}, \mathcal{F})$ is a torsion pair in $\mathcal{Q}$ which proves the statement. It remains to show that $({}^\perp \mathcal{Q} \ast \mathcal{T}, \mathcal{F} )$ is a torsion pair in $\mathcal{L}$. Since $\mathcal{L}$ is abelian, it suffices to show that  $$\begin{array}{rl}
{}^\perp \mathcal{Q} \ast \mathcal{T} =& \{M \in \mathcal{R} \ | \ \Hom_\mathcal{L}(M, \mathcal{F} )=0 \} \\ \mathcal{F}= & \{ N \in \mathcal{L}\  |\  \Hom_\mathcal{L}({}^\perp \mathcal{Q} \ast\mathcal{T}, N) =0 \}
\end{array} $$ Let $N \in \mathcal{L}$ be such that $\Hom_\mathcal{L}({}^\perp \mathcal{Q} \ast \mathcal{T}, N )=0$. In particular, we have that $\Hom_\mathcal{L}({}^\perp \mathcal{Q} , N )=0$ thus $N \in \mathcal{Q}$ and as $0=\Hom_\mathcal{L}(\mathcal{T}, N )= \Hom_\mathcal{Q}(\mathcal{T}, N  )$, $N \in \mathcal{F}$. Now let $M \in \mathcal{L}$ be such that  $\Hom_\mathcal{L}(M, \mathcal{F}) =0 $ and consider the $({}^\perp\mathcal{Q}, \mathcal{Q})$-canonical short exact sequence of $M$
\[ \begin{tikzcd} 0 \arrow[r] & _{{}^\perp\mathcal{Q}} M \arrow[r] & M \arrow[r, "p"] & M_{\mathcal{Q}} \arrow[r] & 0. \end{tikzcd} \] Observe that $\Hom_\mathcal{L}(M_{\mathcal{Q}}, \mathcal{F} )= \Hom_\mathcal{Q}(M_{\mathcal{Q}}, \mathcal{F})=0$, else by pre-composing with the epimorphism $p$ we would obtain a non-zero morphism $M \to \mathcal{F}$. Thus $M_{\mathcal{Q}} \in \mathcal{T}$ and the sequence shows $N$ is an element of ${}^\perp \mathcal{Q} \ast \mathcal{T}$. 
\end{proof}

In the case of torsion pairs in small module categories over artin algebras, which are abelian, there is a well-known symmetry:
\begin{prop} \label{smalo}
\cite{smalo1984torsion} Let $\mathcal{A}\cong \mathsf{mod}\Lambda$ with $\Lambda$ an artin algebra be an abelian category and $(\mathcal{T}, \mathcal{F})$ be a torsion pair in $\mathcal{A}$. Then $\mathcal{T}$ is functorially finite in $\mathcal{A}$ if and only if $\mathcal{F}$ is functorially finite in $\mathcal{A}$.
\end{prop}

We will see that this symmetry extends to the quasi-abelian setting. We note that $\mathcal{L}_{\mathcal{Q}} \cong \mathsf{mod} \Lambda$ for an artin algebra $\Lambda$ is satisfied, for instance, the conditions of Theorem \ref{Lsmallmod} are met and $\mathsf{End}_{\mathcal{Q}}(P)$ is an artin algebra. In particular, if $\mathcal{Q}$ is a $k$-linear hom-finite quasi-abelian category for a field $k$, then $\mathsf{End}_{\mathcal{Q}}(P)$ is finite dimensional and hence artin. The author thanks Lidia Angeleri-H\"{u}gel for pointing out an inaccuracy in a previous version of the following result. 

\begin{lem} \label{Renoughproj}
Let $\mathcal{Q}$ be a quasi-abelian category such that $ \mathcal{L}_{\mathcal{Q}} \cong \mathsf{mod}\Lambda$ for an artin algebra $\Lambda$. Then $\mathcal{Q}$ is functorially finite in $\mathcal{L}_\mathcal{Q}$.
\end{lem}
\begin{proof}
As we noted in Lemma \ref{QinL}, by \cite[Proposition B.3]{bondal2003generators} $\mathcal{Q}$ is a cotilting torsionfree class of $\mathcal{L}_\mathcal{Q}$. Then we may apply \cite[Theorem]{smalo1984torsion} to see that $\mathcal{Q}$ is functorially finite.
\end{proof}

\begin{prop} \label{ffqab} Let $\mathcal{Q}$ be a quasi-abelian category such that $\mathcal{L}_{\mathcal{Q}} \cong  \mathsf{mod}\Lambda$ for an artin algebra $\Lambda$. and let $(\mathcal{T}, \mathcal{F})$ be a torsion pair in $\mathcal{Q}$. Then $\mathcal{T}$ is functorially finite if and only if $\mathcal{F}$ is functorially finite.
\end{prop}
\begin{proof} Let $(\mathcal{T}, \mathcal{F})$ be a torsion pair in a quasi-abelian category $\mathcal{Q}$ and suppose that $\mathcal{F}$ is functorially finite in $\mathcal{Q}$. We begin by showing that $\mathcal{F}$ is functorially finite in $\mathcal{L}=\mathcal{L}_{\mathcal{Q}}$ (this was not done in Proposition \ref{fftorsbij}). As $\mathcal{F}$ is a torsionfree class in $\mathcal{L}$, it is covariantly finite in $\mathcal{L}$. We now show that every $X\in \mathcal{L}$ admits a right $\mathcal{F}$-approximation. Let $ Q \to X$ be a right $\mathcal{Q}$-approximation of $X$, which exists by Lemma \ref{Renoughproj}, and let $F \to Q$ be a right $\mathcal{F}$-approximation of $Q$, which exists by assumption. Then it is easily verified that the composition $F \to X$ is a right $\mathcal{F}$-approximation of $X$. 
Thus $\mathcal{F}$ is a functorially finite torsion class in $\mathcal{A}$ and therefore so is its associated torsion free class ${}^{\perp}\mathcal{Q} \ast T$ in $\mathcal{L}$ by Theorem \ref{Lsmallmod} and Proposition \ref{smalo}. It now follows from Proposition  \ref{fftorsbij}(a) that $\mathcal{T}$ is functorially finite in $\mathcal{Q}$.

For the converse, we reverse the argument but use Proposition \ref{fftorsbij}(b) to see that ${}^{\perp}\mathcal{Q} \ast \mathcal{T}$ is functorially finite in $\mathcal{L}$ which we may do since $\mathcal{L}$ has enough projectives by Corollary \ref{projgen} as $\mathcal{Q}$ has a strong projective generator by assumption.
\end{proof}

\begin{lem} \label{intersection} The intersection of torsion classes in a quasi-abelian category is again quasi-abelian. 
\end{lem}
\begin{proof}
This property is immediately inherited from the abelian setting. 
\end{proof}
\section{Harder-Narasimhan filtrations} \label{hnfilt}
In this section we apply the results of the previous sections to show the existence of Harder-Narasimhan filtrations arising from chains of torsion classes in quasi-abelian categories. We recall the pseudometric on the space of chains of torsion classes defined by \cite{treffinger2018algebraic} building on \cite{bridgeland2007stability} and then investigate topological properties of this space. To begin, we recall the necessary concepts following \cite[\S 2]{treffinger2018algebraic} where it was shown that abelian categories admit Harder-Narasimhan filtrations.

\begin{defn}\label{CTA} For a pre-abelian category, $\mathcal{C}$,  consider the order reversing functions of posets, $\eta$, from the real interval $[0,1]$ to the set of all torsion classes of $\mathcal{C}$ such that $\eta(0) = \mathcal{C}$ and $\eta(1) = 0$. Equivalently, the data of such a map is a chain of torsion classes in $\mathcal{C}$
\[ \eta: \quad 0= \mathcal{T}_1 \subseteq \dots \subset \mathcal{T}_r \subset
\dots \subseteq \mathcal{T}_0 = \mathcal{C} \] with  $r \in [0,1]$ satisfying $\mathcal{T}_r \subseteq \mathcal{T}_{r'}$ if and only if $r \geq r'$. We call such an $\eta$ \emph{quasi-Noetherian} (resp. \emph{weakly-Artinian}) if for every interval $(a,b) \subset [0,1]$ there exists $s \in (a,b)$ such that ${}_{\mathcal{T}_r}M \subset {}_{\mathcal{T}_s}M $ ( resp. $s'\in (a,b)$ such that ${}_{\mathcal{T}_s}M \subset {}_{\mathcal{T}_r}M $) for all $r \in (a,b)$ and $M \in \mathcal{C}$. By $\mathfrak{T}(\mathcal{C})$ we denote the set of all $\eta$ that are quasi-Noetherian and weakly-Artinian.
\end{defn}

 \begin{notation} Let $\eta = (\mathcal{T}_i)_{i \in [0,1]} \in \mathfrak{T}(\mathcal{C})$. For any $j \in [0,1]$, by $\mathcal{F}_j$ we denote the associated torsionfree class of $\mathcal{T}_j$ in $\mathcal{C}$. 
\end{notation}

\begin{lem} \label{uniontors}  Let $\mathcal{Q}$ be a quasi-abelian category and $\eta = (\mathcal{T}_i)_{i \in [0,1]} \in \mathfrak{T}(\mathcal{Q})$. Then for every $r \in [0,1]$, the pairs of subcategories
\[ \Big( \bigcup_{s>r} \mathcal{T}_s, \bigcap_{s>r} \mathcal{F}_s \Big) \ \text{and} \   \Big( \bigcap_{s<r} \mathcal{T}_s , \bigcup_{s<r} F_s \Big) \] are torsion pairs in $\mathcal{Q}$. Moreover, for all $\mathcal{X}\in \mathsf{tors}\mathcal{Q}$, if $\mathcal{X}\subset \mathcal{T}_s$ for all $s<r$ then $\mathcal{X} \subset \bigcap_{s<r} \mathcal{T}_s$. Similarly, if  $ \mathcal{T}_s \subset \mathcal{X}$ for all $s>r$ then $ \bigcup_{s>r} \mathcal{T}_s \subset \mathcal{X}$. 
\end{lem} \begin{proof} 
Let $r \in [0,1]$, we will show that the pair $(\mathcal{T}, \mathcal{F}) = ( \bigcup_{s>r} \mathcal{T}_s, \bigcap_{s>r} \mathcal{F}_s )$ satisfies the hom-orthogonality conditions of Proposition \ref{qabaddtors}.  Since, for all $s>r$, $\mathcal{F} \subset \mathcal{F}_s$ we have that $\mathsf{Hom}_\mathcal{Q}(\mathcal{T}_{s}, \mathcal{F}) = 0$ and hence $\mathsf{Hom}_\mathcal{Q}(\mathcal{T}, \mathcal{F}) = 0$. 

Let $Y \in \mathcal{Q}$ be such that $\mathsf{Hom}_\mathcal{Q}(\mathcal{T}, Y) = 0$. Then for all $s>r$, $\mathsf{Hom}_\mathcal{Q}(\mathcal{T}_s, Y) =0$ and so $Y \in \mathcal{F}_s$ for all $s>r$, thus $Y \in \mathcal{F}$.

Let $X \in \mathcal{Q}$
and suppose that $X \not \in \mathcal{T}$. Then for all $s>r$, $ X \not \in \mathcal{T}_s$. Since $(\mathcal{T}_s, \mathcal{F}_s)$ is a torsion pair in $\mathcal{Q}$ it follows that $\mathsf{Hom}_\mathcal{Q}(X, \mathcal{F}_s) \not = 0$ for all $s>r$. We deduce that $\mathsf{Hom}_\mathcal{Q}(X, \mathcal{F}) \not = 0$ as the $\mathcal{F}_i$ are ordered by inclusion. 

Thus we have shown that $\mathcal{T} = \{ X \in \mathcal{Q} \mid \mathsf{Hom}_\mathcal{Q}(X, \mathcal{F} \} =0$ and $\mathcal{F} = \{ X \in \mathcal{Q} \mid \mathsf{Hom}_\mathcal{Q}(\mathcal{T}, X ) =0 \}$, so by Proposition \ref{qabaddtors}, $(\mathcal{T}, \mathcal{F})$ is a torsion pair in $\mathcal{Q}$. The fact that $(\bigcap_{s<r} \mathcal{T}_s , \bigcup_{s<r} F_s )$ is a torsion pair follows from a similar argument and the remaining claims are basic set theory. 
\end{proof}

\begin{defn} 
Let $\mathcal{C}$ be a pre-abelian category and $\eta = (\mathcal{T}_i)_{i \in [0,1]} \in \mathfrak{T}(\mathcal{C})$. Define the subcategories $\mathcal{P}^\eta_r$ as follows 
\[ \mathcal{P}^{\eta}_r  = \left\lbrace \begin{array}{ll} \bigcap_{s>0} \mathcal{F}_s & \text{if } r=0 \\ \Big( \bigcap_{s<r} \mathcal{T}_s \Big) \cap   {\Big( \bigcap_{s>r} \mathcal{F}_s \Big)} & \text{if } r \in (0,1) \\ \bigcap_{s<1} \mathcal{T}_s & \text{if } r=1 \end{array} \right.  \] 
\end{defn}

\begin{rem} In a quasi-abelian category $\mathcal{Q}$, for every $\eta = (\mathcal{T}_i)_{i \in [0,1]} \in \mathfrak{T}(\mathcal{Q})$, each $\mathcal{P}^\eta_r$ is quasi-abelian. For $r=0,1$ this is obvious. For $r \in (0,1)$ observe  that 
\[ \bigcup_{s>r} \mathcal{T}_s \subseteq \mathcal{T}_r \subseteq  \bigcap_{s<r} \mathcal{T}_s  \] thus  $\bigcup_{s>r} \mathcal{T}_s $ and $ \bigcap_{s<r} \mathcal{T}_s $ define twin torsion pairs with heart $\mathcal{P}^{\eta}_r$ and is hence quasi-abelian by Theorem \ref{qab}. We also note that for all twin torsion pairs $[(\mathcal{C}, \mathcal{D}), (\mathcal{C}', \mathcal{D}')]$ in $\mathcal{Q}$ their heart, $\mathcal{C}' \cap \mathcal{D}$, appears as $\mathcal{P}^{\eta}_r$ in the chain of torsion classes
\[ \eta: \quad  0 \subset \mathcal{C} \subset \mathcal{C}' \subset \mathcal{Q}  \]  for some $r \in [0,1]$.
\end{rem}

\begin{setup} \label{setup1} Let $\mathcal{A}$ be an abelian category and fix twin torsion pairs  $[(\mathcal{C}, \mathcal{D})$, $(\mathcal{C}', \mathcal{D}')]$ in $\mathcal{A}$ and set $\mathcal{Q}=\mathcal{C}' \cap \mathcal{D}$. 
By Theorem \ref{torsbij1}, we may identify $\mathfrak{T}(\mathcal{Q})$ bijectively with a subset of  $\mathfrak{T}(\mathcal{A})$ along the map \begin{align*}
\phi_{\mathcal{C}}=\phi : \mathfrak{T}(\mathcal{Q}) & \hookrightarrow  \mathfrak{T}(\mathcal{A}) \\ \eta = (\mathcal{T}_i)_{i \in [0,1]} & \mapsto  \phi(\eta) = (\mathcal{X}_i)_{i \in [0,1]} 
\end{align*}
where 
\[ \mathcal{X}_i  = \left\lbrace \begin{array}{ll} \mathcal{A} & \text{if } i =0 \\ \mathcal{C} \ast \mathcal{T}_{i} & \text{if } i \in (0,1) \\ 0 & \text{if } i =1. \end{array} \right.  \] We denote the image of $\phi_{\mathcal{C}}$ by $\mathfrak{T}_{\mathcal{C}}(\mathcal{Q})$. Thus $\mathfrak{T}_{\mathcal{C}}(\mathcal{Q})$ consists of all $\eta = (\mathcal{T}_i)_{i \in [0,1]} \in \mathfrak{T}(\mathcal{A})$ such that $\mathcal{C} \subseteq \mathcal{T}_i \subseteq \mathcal{C}'$ for all $i \in (0,1)$. We remark that, in light of Remark \ref{twoofhearts}, this map does indeed depend on $\mathcal{C}$ (since then $\mathcal{Q}$ determines $\mathcal{C}'$ by Theorem \ref{torsbij1}). 
\end{setup}

We investigate the subcategories $\mathcal{P}^{\phi(\eta)}_r$.
\begin{lem} \label{slicings}
In the situation of Set-up \ref{setup1}. Let $\eta = (\mathcal{T}_i)_{i \in [0,1]} \in \mathfrak{T}(\mathcal{Q})$. Then 
\[  \mathcal{P}^{\phi(\eta)}_r = \left\lbrace \begin{array}{ll} \mathcal{P}^{\eta}_0 \ast \mathcal{D}'  & \text{if } r =0 \\  \mathcal{P}^{\eta}_{r} & \text{if } r \in (0,1) \\ \mathcal{C} \ast \mathcal{P}^{\eta}_1 & \text{if } r=1. \end{array} \right.  \]
\end{lem}
\begin{proof}
Let $r \in (0,1)$, and observe
\begin{align*}
\mathcal{P}^{\phi(\eta)}_r   = & \Bigg( \bigcap_{s \in (0, r)} (\mathcal{C}\ast \mathcal{T}_{s}) \Bigg) \cap   {\Bigg( \bigcup_{s\in(r,1)} (\mathcal{C}\ast \mathcal{T}_{s}) \Bigg)}^{\perp} \\ = &  \Bigg( \mathcal{C}\ast \bigcap_{s \in (0, r)} \mathcal{T}_{s} \Bigg) \cap {\Bigg( \mathcal{C}\ast \bigcup_{s\in(r,1)}  \mathcal{T}_{s} \Bigg)}^{\perp}  \\ = &\Bigg( \mathcal{C}\ast \bigcap_{s \in (0, r)} \mathcal{T}_{s} \Bigg) \cap   {\Bigg( \Big( \bigcup_{s\in(r,1)} \mathcal{T}_{s} \Big)^{\perp_{\mathcal{Q}}} \ast \mathcal{D}' \Bigg)} \\ =&  \Big( \bigcap_{s \in (0, r)} \mathcal{T}_{s} \Big) \cap    \Big( \bigcup_{s\in(r,1)} \mathcal{T}_{s} \Big)^{\perp_\mathcal{Q}} = \mathcal{P}^{\eta}_{r} \end{align*} where the first and last equalities follow from Lemma \ref{uniontors} and the definitions knowing that $\mathcal{T}_1 = 0$ and $\mathcal{T}_0 =\mathcal{Q}$. The second equality is straightforward set theory, the third equality follows from Theorem \ref{torsbij1} and the fourth equality holds by Lemma \ref{reflecthearts}.  The cases $r=0, 1$ follow by similar arguments.
\end{proof}

We now show that every $\eta \in \mathfrak{T}(\mathcal{Q})$ induces a unique Harder-Narasimhan filtration of each object.

\begin{thm} \label{thm1} In the situation of Set-up \ref{setup1}.
Let $\eta = (\mathcal{T}_i)_{i \in [0,1]} \in \mathfrak{T}(\mathcal{Q})$. Then for all $M \in Q$ there exists a unique (up to isomorphism) Harder-Narasimhan filtration of $M$ with respect to $\eta$ in $\mathcal{Q}$. That is, a filtration 
\[ 0 = M_0 	\subsetneq M_1 	\subsetneq \dots 	\subsetneq M_n = M \] of $M$ in $\mathcal{Q}$ such that 
\ben
\item[(HN1)] $M_k / M_{k-1} \in \mathcal{P}^{\eta}_{r_k}$ for all $1\leq k \leq n$.
\item[(HN2)] $r_k > r_{k'}$ if and only if $k< k'$.
\ee

\end{thm}
\begin{proof}
Let $M \in \mathcal{Q}$ and  $\eta = (\mathcal{T}_i)_{i \in [0,1]} \in \mathfrak{T}(\mathcal{Q})$. 
Let 
\[ 0 = M_0 \subset M_1 \subset \dots \subset M_n = M \] be the Harder-Narasimhan filtration of $M$  with respect to $\phi(\eta)$ in $\mathcal{A}$, which exists by \cite[2.9]{treffinger2018algebraic}, so that for all $1 \leq k, k' \leq n$, $M_k / M_{k-1} \in \mathcal{P}^{\phi(\eta)}_{r_k}$ and $r_k > r_{k'}$ precisely when $k < k'$. We claim that this is also  the Harder-Narasimhan filtration of $M$ with respect to $\eta$ in $\mathcal{Q}$. 

We first show that for all $1 \leq k \leq n$, $M_k / M_{k-1} \in \mathcal{P}^{\eta}_{r_k}$. When $r_k \in (0,1)$, this is trivially true by Lemma \ref{slicings}. It remains to check for $r_k =0,1$. Observe that the only case where $r_k = 0$ (resp. $r_k = 1$) can occur is when $k=n$ (resp. $k=1$). So suppose that $r_n=0$, then $M_n / M_{n-1} = M / M_{n-1} \in  \mathcal{P}^{\phi(\eta)}_{0}   = \mathcal{P}^{\eta}_0 \ast \mathcal{D}'$. As $M \in \mathcal{Q} $ is an element of  $\mathcal{C}'$, so is the quotient $ M / M_{n-1}$. Thus $ M / M_{n-1} \in \mathcal{C}' \cap (\mathcal{P}^{\eta}_0 \ast \mathcal{D}') = \mathcal{P}^{\eta}_0$ by Theorem \ref{torsbij1}. 
Similarly, we see that $M_1 / M_0 = M_1 \in \mathcal{D} \cap \mathcal{P}^{\phi(\eta)}_1 = \mathcal{D} \cap (\mathcal{C} \ast \mathcal{P}^{\eta}_1)= \mathcal{P}^{\eta}_1$. By Lemma \ref{slicings}, this implies that (HN1) holds. Note that (HN2) holds since it is inherited from the abelian case $\mathcal{A}$ as is the uniqueness of the filtration up to isomorphism.

It remains to show that $M_i \in \mathcal{Q}$ for all $1 \leq k \leq n $. We proceed by induction on $k$. For $k=1$ we have shown that $M_1 = M_1 / M_0 \in  \mathcal{P}^{\eta}_{r_1} \subset \mathcal{Q}$ for some $r_1 \in [0,1]$. The $k>1$ case follows by using the short exact sequences 
\[ \begin{tikzcd} 0 \arrow[r] & M_{k-1} \arrow[r] & M_k \arrow[r] & M_k / M_{k-1} \arrow[r] & 0 \end{tikzcd} \] as $M_k / M_{k-1} \in \mathcal{P}^{\eta}_{r_k} \subset \mathcal{Q}$ and since $\mathcal{Q}$ is closed under extensions. 
\end{proof}

\begin{cor} \label{cor1}
 Let $\mathcal{Q}$ be a quasi-abelian category and $\eta = (\mathcal{T}_i)_{i \in [0,1]} \in \mathfrak{T}(\mathcal{Q})$. Then for all $M\in \mathcal{Q}$ there exists a unique (up to isomorphism) Harder-Narasimhan filtration with respect to $\eta$ in $\mathcal{Q}$.
\end{cor}
\begin{proof} The result follows from Theorem \ref{thm1} as $\mathcal{Q}$ appears as $\mathcal{P}^\eta_r$ for some $r \in [0,1]$ in the chain of torsion classes
\[ \eta: \quad  0 \subset {}^{\perp}\mathcal{Q} \subset \mathcal{L}_\mathcal{Q} \]  in $\mathfrak{T}(\mathcal{L}_\mathcal{Q})$.
\end{proof}

We recall that, by \cite[6.1]{bridgeland2007stability} and \cite[7.1]{treffinger2018algebraic}, for an abelian category $\mathcal{A}$, $\mathfrak{T}(\mathcal{A})$ is a topological space with pseudometric given by 
\begin{equation} \label{topspace} 
d(\eta, \eta') = \mathsf{inf}\{\varepsilon \in [0,1] \ | \ \mathcal{P}^{\eta'}_r \subset \mathcal{P}^{\eta}_{[r-\varepsilon, r +\varepsilon]} \forall r \in [0,1] \}
\end{equation}for $\eta, \eta' \in \mathfrak{T}(\mathcal{A})$. Where 
\[ \mathcal{P}^{\eta}_{[a,b]} := \mathsf{Filt}\Big( \bigcup_{s \in [a,b]} \mathcal{P}^{\eta}_s \Big) 
\] for $0 \leq a \leq b \leq 1$ and we set $\mathcal{P}^\eta_r=0$ for all $r\not\in [0,1]$.

\begin{rem} \label{zero distance} Note that, for $\eta, \eta' \in \mathfrak{T}(\mathcal{A})$, we have $d(\eta, \eta') = 0 $ if and only if $\mathcal{P}^{\eta}_r = \mathcal{P}^{\eta'}_r$ for all $r \in [0,1]$. \end{rem}

As we remarked earlier, the embedding of $\mathfrak{T}(\mathcal{Q})$ (of Set-up \ref{setup1}) in $\mathfrak{T}(\mathcal{A})$ depends on $\mathcal{C}$. So when $\mathcal{Q}$ occurs as the heart of many twin torsion pairs, $\mathfrak{T}(\mathcal{Q})$ can be embedded into $\mathfrak{T}(\mathcal{A})$ in as many ways. To finish, we see  that the various embeddings of $\mathcal{Q}$ are at maximal distance apart in $\mathfrak{T}(\mathcal{A})$ and that each of these embeddings is closed.

\begin{thm}
In the situation of Set-up \ref{setup1}, $\mathfrak{T}_{\mathcal{C}}(\mathcal{Q})$ is a closed set of the topological space $\mathfrak{T}(\mathcal{A})$ and if $\mathcal{Q}$ also occurs as the heart of different twin torsion pairs $[(\mathcal{C}_1, \mathcal{D}_1)$, $(\mathcal{C}_1', \mathcal{D}_1')]$, then  \[ d(\mathfrak{T}_{\mathcal{C}}(\mathcal{Q}), \mathfrak{T}_{\mathcal{C}_1}(\mathcal{Q}))=1 \] where $\mathfrak{T}_{\mathcal{C}_1}(\mathcal{Q})$ is defined following Set-up \ref{setup1}.
\end{thm}

\begin{proof}
First, we show that $\mathfrak{T}_{\mathcal{C}}(\mathcal{Q})$ contains all of its accumulation points and is therefore a closed set of $\mathfrak{T}(\mathcal{A})$. To this end, let  $\eta = (\mathcal{T}_i)_{i \in [0,1]} \in \mathfrak{T}(\mathcal{A})$ such that there exists $\eta' \in \mathfrak{T}_{\mathcal{C}}(\mathcal{Q})$ with $d(\eta, \eta')= d(\eta', \eta)=0$. Then by,   Remark \ref{zero distance}, for all $r\in [0,1]$, $\mathcal{P}^{\eta}_r = \mathcal{P}^{\eta'}_r$. In particular, $\mathcal{P}^{\eta}_1 \subseteq \mathcal{C}$ and $\mathcal{P}^{\eta}_0 \subseteq \mathcal{D}'$, thus $\mathcal{C} \subseteq \mathcal{T}_i \subseteq \mathcal{C}'$ for all $i \in (0,1)$ and we conclude that $\eta \in \mathfrak{T}_{\mathcal{C}}(\mathcal{Q})$. 

Now suppose that $\mathcal{Q} = \mathcal{C}_1' \cap \mathcal{D}_1$ for some twin torsion pair $[(\mathcal{C}_1, \mathcal{D}_1)$, $(\mathcal{C}_1', \mathcal{D}_1')]$ with $\mathcal{C} \not = \mathcal{C}_1$. We show that for all $\eta \in \mathfrak{T}_{\mathcal{C}}(\mathcal{Q})$ and $\eta_1 \in \mathfrak{T}_{\mathcal{C}_1}(\mathcal{Q})$, $d(\eta, \eta_1)=1$. By  the formula (\ref{topspace}), it is enough to show the existence of some \begin{equation} \label{distance}
0\not =X \in (\mathcal{P}^{\eta}_0 \backslash \mathcal{P}^{\eta_1}_0) \cap (\mathcal{P}^{\eta_1}_1 \backslash \mathcal{P}^{\eta}_1).
\end{equation}  
Let $Y \in \mathcal{D}'_0 \backslash \mathcal{D}'_1$, we claim that $X := {}_{\mathcal{C}_1}Y$ satisfies (\ref{distance}). Clearly, $X  \in \mathcal{C}_1 \subseteq \mathcal{P}^{\eta_1}_1$ and, as $\mathcal{D}'_0$ is closed under subobjects, $X \in \mathcal{D}'_0 \subseteq \mathcal{P}^{\eta}_0$. 

We now show that $X \not\in \mathcal{P}^{\eta_1}_0 = \mathcal{D}_1' \ast \mathcal{P}^{\phi_{\mathcal{C}_{1}}^{-1}(\eta_1)}_0$. Observe that as $X\in \mathcal{C}_1'$, $X_{\mathcal{D}'_1}=0$, thus, by Lemma \ref{starstar}, $X \in  \mathcal{P}^{\eta_1}_0$ if and only if $X \in \mathcal{P}^{\phi_{\mathcal{C}_{1}}^{-1}(\eta_1)}_0 \subset Q$. But as $X \in \mathcal{C}_1$, $X \not\in \mathcal{Q}$ and in particular, $X \not\in \mathcal{P}^{\eta_1}_0$. Similarly, one verifies that $X \not \in \mathcal{P}^{\eta_0}_1$ and we are done. 
\end{proof}
\bibliographystyle{halpha-abbrv}
\bibliography{biv}
  \par
  \medskip
  \begin{tabular}{@{}l@{}}%
    \textsc{School of Mathematics, University of Leicester, Leicester,} \\ \textsc{le1 7rh, uk}\\
    \texttt{ast20@le.ac.uk}
  \end{tabular}
\end{document}